\documentclass[twoside]{article}[11pt]

\usepackage{amssymb,amsmath}
\usepackage{amsthm}
\usepackage[colorlinks,linkcolor=violet,citecolor=violet]{hyperref}
\usepackage[nameinlink]{cleveref}
\usepackage{float}
\usepackage{fullpage}
\usepackage[dvipsnames]{xcolor}
\usepackage{fancyhdr,datetime}
\usepackage{graphicx}
\usepackage{subcaption}
\usepackage{tikz-cd}
\usepackage{rotating}
\usepackage{etoolbox}
\usepackage{mathrsfs}
\usepackage[percent]{overpic}
\usepackage[export]{adjustbox}
\usepackage{stmaryrd}
\usepackage{lipsum}
\usepackage{amsfonts}
\usepackage{graphicx}
\usepackage{epstopdf}
\usepackage{algorithmic}
\usepackage[baseline]{euflag}
\usepackage{tikz}
\usepackage{rotating}
\usepackage{colortbl}
\usepackage[table]{xcolor}
\usepackage{multirow}
\usepackage{float}
\usepackage{subcaption}
\usepackage{pgfplots}
\usepackage{enumitem}
\usepackage{amsopn}
\usepackage{mathtools}

\ifpdf
\DeclareGraphicsExtensions{.eps,.pdf,.png,.jpg}
\else
\DeclareGraphicsExtensions{.eps}
\fi

\setlist[enumerate]{leftmargin=.5in}
\setlist[itemize]{leftmargin=.5in}

\renewcommand\footnoterule{\kern-3pt \hrule width \textwidth \kern 2.6pt}

\let\oldsection\section 
\renewcommand{\section}{\renewcommand{\theequation}{\thesection.\arabic{equation}}\oldsection}

\newtheorem{lemma}{Lemma}[section]
\newtheorem{example}[lemma]{Example}

\newtheorem{definition}[lemma]{Definition}
\newtheorem{proposition}[lemma]{Proposition}
\newtheorem{remark}[lemma]{Remark}
\newtheorem{theorem}[lemma]{Theorem}

\newtheorem{corollary}[lemma]{Corollary}

{\scshape}{\slshape}

\newcommand{\italgf}{\slshape  }

\DeclareMathOperator{\trace}{\rm{Trace}}

\newcommand{\transpose}[1]{\ensuremath{{#1}^{\mathrm{T}}}}

\def\supp{\hbox{\rm{supp}}}

\def\C{\mathbb{C}}
\def\R{\mathbb{R}}
\def\N{\mathbb{N}}

\newcommand{\var}[1]{\ensuremath{x_{#1}}}
\newcommand{\RX}{\ensuremath{\R[\var{}]}}
\newcommand{\RZ}{\ensuremath{\R[z]}}
\newcommand{\CX}{\ensuremath{\C[\var{}]}}
\newcommand{\rspan}[1]{\ensuremath{\mathrm{span}_\R #1 }}

\newcommand{\lineargroup}{\ensuremath{\mathcal{G}}}
\newcommand{\symmetricgroup}[1]{\ensuremath{\mathfrak{S}_{#1}}}
\newcommand{\irreps}{\ensuremath{h}}
\newcommand{\reynolds}{\ensuremath{\mathcal{R}^\lineargroup}}
\newcommand{\irrepdim}[1]{\ensuremath{d^{(#1)}}}
\newcommand{\gmultiplicity}[2]{\ensuremath{m^{(#1)}_{#2}}}
\newcommand{\isotypicprojection}[2]{\ensuremath{\mathbf{P}^{(#1)}_{#2}}}

\newcommand{\posmat}{\ensuremath{\mathbf{P}}}

\newcommand{\quadraticmodule}[1]{\ensuremath{\mathrm{QM}(\mathbf{#1})}}
\newcommand{\gmoment}[1]{\ensuremath{y_{#1}}}
\newcommand{\gmomentsequence}{\ensuremath{\mathbf{y}}}
\newcommand{\gmomentmatrix}[2]{\ensuremath{\mathbf{M}^\lineargroup_{#2}(#1)}}
\newcommand{\momentblock}[3]{\ensuremath{\mathbf{M}^{(#2)}_{#3}(#1)}}
\newcommand{\sosmatrix}[2]{\ensuremath{\mathbf{Q}^{(#1)}_{#2}}}
\newcommand{\soscone}[2]{\ensuremath{\Sigma(\rspan{\mathcal{S}^{(#1)}_{#2}})^2}}
\newcommand{\sdpmatrix}[2]{\ensuremath{\mathbf{A}^{(#1)}_{#2}}}
\newcommand{\sdpcone}[2]{\ensuremath{\mathrm{Sym}^{(#1)}_{#2}}}

\newcommand{\pretspmatrix}[2]{\ensuremath{\tilde{\mathbf{B}}^{(#1)}_{#2}}}
\newcommand{\tspmatrix}[2]{\ensuremath{\mathbf{B}^{(#1)}_{#2}}}
\newcommand{\tspsupport}[2]{\ensuremath{\mathcal{B}^{(#1)}_{#2}}}

\definecolor{MarkerColour}{HTML}{B6073F}

\newcommand{\greencell}{\cellcolor[RGB]{107,188,91}}

\pgfplotsset{compat=1.18}
\definecolor{coral}{RGB}{255,127,80}

\title{Exploiting Term Sparsity in Symmetry-Adapted Basis\\for Polynomial Optimization}

\author{	
	Igor~Klep\thanks{University of Ljubljana, University of Primorska} , 
	Victor~Magron\thanks{LAAS-CNRS Toulouse} \thanks{Institute of Mathematics Toulouse} , 
	Tobias~Metzlaff$^\dagger$, 
	Jie~Wang\thanks{AMSS-CAS}
}

\newcommand{\Addresses}{
\bigskip
\footnotesize
I.~Klep, 
\textsc{University of Ljubljana, Faculty of Mathematics and Physics 
	\& University of Primorska, Famnit, Koper 
	\& Institute of Mathematics, Physics and Mechanics, Ljubljana, Slovenia}
\par\nopagebreak
\textit{E-mail address}: \href{mailto:igor.klep@fmf.uni-lj.si}{igor.klep@fmf.uni-lj.si}
\par\nopagebreak
\textit{Webpage}: \url{https://igorklep.codeberg.page/}

\medskip

V.~Magron, 
\textsc{LAAS CNRS \& Institute of Mathematics, Toulouse, France}
\par\nopagebreak
\textit{E-mail address}: \href{mailto:vmagron@laas.fr}{vmagron@laas.fr}
\par\nopagebreak
\textit{Webpage}: \url{https://homepages.laas.fr/vmagron/}

\medskip

T.~Metzlaff (Corresponding Author), 
\textsc{LAAS-CNRS, Toulouse, France}
\par\nopagebreak
\textit{E-mail address}: \href{mailto:math@tobiasmetzlaff.com}{math@tobiasmetzlaff.com}
\par\nopagebreak
\textit{Webpage}: \url{https://tobiasmetzlaff.com/}
\par\nopagebreak
\textit{ORCID}: \texttt{0000-0002-0688-7074}

\medskip

J.~Wang, 
\textsc{State Key Laboratory of Mathematical Sciences, Academy of Mathematics and Systems Science, Chinese Academy of Sciences, Beijing, China}
\par\nopagebreak
\textit{E-mail address}: \href{mailto:wangjie212@amss.ac.cn}{wangjie212@amss.ac.cn}
\par\nopagebreak
\textit{Webpage}: \url{https://wangjie212.github.io/jiewang/}
}

\begin{document}

\maketitle

\begin{abstract}
Polynomial optimization problems are infinite-dimensional, nonconvex, NP-hard, 
and are often handled in practice with the moment-sums of squares hierarchy of semidefinite programming bounds. 
We consider problems where the objective function and constraint polynomials are 
invariant under the action of a finite group. 
The present paper simultaneously exploits group symmetry and term sparsity 
in order to reduce the computational cost of the hierarchy. 
We first exploit symmetry by writing the semidefinite matrices in a symmetry-adapted basis according to an isotypic decomposition. 
The matrices in such a basis are block diagonal. 
Secondly, we exploit term sparsity on each block to further reduce the optimization matrix variables. 
This is a non-trivial extension of the term sparsity-based hierarchy related to sign symmetry that was introduced by two of the authors. 
Our method is compared with existing techniques via benchmarks on quartics with 
dihedral, cyclic and symmetric group symmetry. 
\end{abstract}

\thispagestyle{empty}

\tableofcontents

\section{Introduction}\label{section_1_introduction}

Computational problems can be solved more efficiently 
through the exploitation of algebraic or combinatorial structures 
that emerge from the underlying data. 
In the context of constrained polynomial optimization, 
this article focuses on moment-sums of squares-based semidefinite 
approximation methods. 
We present a moment-sums of squares hierarchy -- also called Lasserre hierarchy -- of lower bounds that simultaneously exploits 
group symmetry and term sparsity. 
We prove convergence to the optimal value under the Archimedean hypothesis (\Cref{thm_sparsity_convergence_1,thm_sparsity_convergence_2}) and 
provide benchmarks to compare with existing methods in practice (\Cref{section_4_benchmarks}). 

\textbf{Background.} 
The objective of polynomial optimization is to minimize 
a given polynomial over a real algebraic or semialgebraic set. 
The problem is nonconvex, infinite-dimensional, and NP-hard. 
To solve it in practice, Lasserre has developed in \cite{lasserre01} 
a hierarchy of semidefinite lower bounds that converges under mild assumptions 
to the optimal value. 
The technique involves, on the primal side, the reformulation 
to a moment problem, and, on the dual side, certifying positivity via sums of squares (SOS). 
The computational implementation results in a semidefinite program (SDP)
and comes from a choice of polynomial basis. 
Originally, Lasserre formulated the approach in the standard monomial basis.
Advanced techniques, such as extraction of minimizers 
via flat extension \cite{laurentmourrain09} or complex variables \cite{josz18,wang2025real}, 
are often formulated in this standard basis as well. 
More on the Lasserre hierarchy can be found in \cite{lasserre09,lasserre21}. 

Throughout, we assume that the objective function and the constraint polynomials 
are invariant under the linear action of a finite group. 
Our strategy begins with the exploitation of this invariance 
by reformulating the problem in a symmetry-adapted basis. 
We strongly emphasize that this does not, at least not in theory, 
change the numerical value of each bound in the Lasserre hierarchy. 
More precisely, the SDP matrices in the underlying 
moment-SOS relaxation are obtained by change of basis
(so that trace and positive semidefiniteness are stable) 
and feature a block-diagonal structure. 
This significantly benefits the time to solve the SDP 
at the cost of a general preprocess of computing the basis and assembling the SDP. 
This preprocess is independent of the objective function and 
can therefore be applied to a class of problems that share the same 
group symmetry, number of variables, and relaxation order. 
Gatermann had already applied symmetry reduction for dynamical systems in \cite{gatermann2000}, 
and, together with Parrilo, given a reduction scheme for SDPs with an indication of 
how to apply this on the level of polynomial optimization \cite{Gatermann04}, 
see also \cite{Scheiderer08,vallentin08}. 
This was later extended by Riener, Theobald, Andr\'en and Lasserre in \cite{riener2013exploiting} 
and to trigonometric optimization by one of the authors in \cite{Metzlaff24}. 


On top of that, we employ term sparsity exploitation. 
If the polynomial is sparse in the symmetry-adapted basis, 
the idea is to remove superfluous entries from the SDP matrices. 
The pioneering papers on sparsity exploitation in polynomial optimization \cite{Waki05,lasserre06} deal with correlative sparsity, 
which is concerned with links between the variables. 
This technique is based on reindexing the matrices in the SDP
by considering subsets of the variables to obtain a hierarchy with quasi block-diagonal SDP matrices, each block having a size related to the cardinality of these subsets, see \cite{Waki16}. 
These patterns usually arise in large scale industrial problems for 
optimal power flow \cite{JoszThesis}, 
roundoff error bound analysis \cite{MagronConstantinidesDonaldson17,Magron18},
and volume approximation \cite{TacchiWeisserLasserreHenrion21}. 

In the present paper, we focus on term sparsity, 
which is concerned with links between the monomials instead of the variables. 
These patterns arise in problems with sign symmetry \cite{lofberg09}, coordinate projection \cite{Parrilo15}, and cross-sparsity \cite{WangLiXia19}, as well as sums of nonnegative circuit polynomials (SONC) \cite{deWolff16} or signomials (SAGE) \cite{Chandrasekaran16}. 
The work \cite{magron21} on term sparsity by two of the authors generalizes these concepts to a term sparsity exploiting hierarchy, where, at each step of the Lasserre hierarchy, 
a sparsity graph is constructed. 
The nodes of the graph are the monomial basis elements and the edges encapsulate their term sparsity pattern (tsp) in the original data. 
Through consecutive support and maximal chordal extensions, 
one obtains another level of the hierarchy, 
which yields the same SDP block-structure as the one obtained after exploiting sign symmetry. 
This level is obtained after finitely many steps at a lower cost. 
Correlative and term sparsities can be successfully combined \cite{magron22}. 
We refer the interested reader to the recent book \cite{magron2023sparse} dedicated to theoretical and practical aspects of both sparsity types. 

\textbf{Contributions.} 
The novelty of the present work is to replace the monomial nodes of the tsp graph with invariant and equivariant elements of the symmetry-adapted basis. 
We show that this replacement preserves the achievement of the dense bound after finitely many steps in \Cref{thm_sparsity_convergence_1} and the asymptotic convergence in \Cref{thm_sparsity_convergence_2} when maximal chordal extensions are chosen. 
Next to these theoretical results, we provide a quantitative comparison with the dense approach, the initial TSSOS reduction, and the pure symmetry reduction. 
The algorithm is fully implemented and running in the Julia package 
\href{https://github.com/wangjie212/TSSOS}{\tt TSSOS}. 
The essence of our algorithm is to 
\begin{enumerate}
\item compute a symmetry-adapted basis; 
\item construct a tsp graph in the symmetry-adapted basis with respect to the entire symmetry group; 
\item perform chordal + support extensions. 
\end{enumerate}
The approach works for any finite group and comes with convergence guarantees. 
Note that step 1. is independent of the objective function 
but depends only on the group action. 

For the case of the symmetric group $\mathfrak{S}_n$, 
we can compare with \cite{JohannesThesis}. 
The cited work offers a complementary approach where the strategy is to
\begin{enumerate}
\item construct a tsp graph with monomial orbits and edge orbits;
\item perform a symmetric chordal extension; 
\item apply symmetry reduction to the so-obtained blocks of the tsp graph with respect to stabilizer subgroups. 
\end{enumerate}
Here, step 2. relies on making an orbit graph simplicial, 
which is, according to the author of \cite{JohannesThesis}, an expensive operation. 
In our case, the most expensive operation is the computation of products of basis elements 
and the evaluation under the Reynolds operator. 
Furthermore, \cite{JohannesThesis} does not guarantee convergence. 

\textbf{Structure of the article.} 
We first review the symmetry reduction scheme in \Cref{section_2_symmetry_reduction} 
and highlight the results on symmetric sums of squares from the literature. 
\Cref{section_3_symmetry_and_sparsity} then introduces 
the tsp graph for the symmetry-adapted basis and 
contains our main theoretical results, accompanied by small examples. 

We conclude our work with several instances of benchmarks in \Cref{section_4_benchmarks}. 
The first scenario is for minimizing the one-dimensional ring Ising quartic 
with dihedral symmetry over a Laplacian ball. 
In this case, all bounds are sharp at the lowest possible relaxation and sparsity order. 
We use maximal chordal extensions to indicate the decrease in block sizes and numbers of constraints. 
The second scenario is the minimization of a two-dimensional torus grid quartic 
over a finite set defined by a single equality constraint. 
Here, we require several levels of sparsity to achieve the optimal bound 
but still beat previous block sizes. 

In \Cref{section_correlative_symmetry}, 
we provide concluding remarks and further investigation tracks. 
In particular, we address the possibility of combining correlative sparsity with orbit space reduction as well as extending our framework to the analysis of equivariant dynamical systems. 


\section{Preliminary Background on Symmetric Sums of Squares}\label{section_2_symmetry_reduction}

For given polynomials 
$f,g_1,\ldots,g_\ell\in \R[\var{1},\ldots,\var{n}] =: \RX$, 
we consider the \textbf{polynomial optimization problem}
\begin{equation}\label{eq_pop}\tag{POP}
    f^*
\coloneqq  \min\{ f(X)\,\vert\,X\in K \}
    \quad \mbox{with} \quad
    K
\coloneqq  \{X\in\R^n\,\vert\, 
    \forall \, 1\leq k\leq \ell: 
    g_k(X)\geq 0\},
\end{equation}
where $K$ is assumed to be compact. 

\subsection{Moment Relaxation and SOS Reinforcement}

A \textbf{sum of squares} (SOS) is a polynomial $q\in\RX$ 
of the form $q=\sum_t q_t^2$ with 
$q_t \in \RX$ and $t$ ranging over a finite index set. 
The set of sums of squares of elements 
in $\RX$ is denoted by $\Sigma\,\RX^2$. 
The \textbf{quadratic module} generated by the 
constraint polynomials $g_1,\ldots,g_\ell$ is 
\begin{equation}
    \quadraticmodule{g} 
\coloneqq  \left\{ 
    q_0 + \sum\limits_{k=1}^\ell q_k\,g_k 
    \,\vert\, 
    q_0,q_1,\ldots,q_\ell \in \Sigma\,\RX^2 
    \right\} . 
\end{equation}
Furthermore, $\quadraticmodule{g}$ is called \textbf{Archimedean} 
if it contains an element $p\in\quadraticmodule{g}$ 
such that the semialgebraic superset 
$\{X\in\R^n\,\vert\,p(X)\geq 0\}$ 
of $K$ is compact. 

A linear form $L: \RX \to \R$ is said to have a 
representing probability measure $\mu$ on $K$ if
$L(p) = \int_K p(X) \, \mathrm{d} \mu( X )$ 
for all $p\in\RX$. 
For the standard monomial basis $\{\var{}^\alpha\,\vert\,\alpha\in\N^n\}$, 
the quantities $y_\alpha \coloneqq L(\var{}^\alpha) = 
\int_K X^\alpha \, \mathrm{d} \mu( X )$ 
are called \textbf{moments}. 
For more information on the theory of positive polynomials, moments, and sums of squares, 
the reader is referred to \cite{marshall08,lasserre21}. 


\begin{theorem}[Putinar's Positivstellensatz]\label{thm_putinar}\cite{putinar93}
Assume that $\quadraticmodule{g}$ is Archimedean. 
\begin{enumerate}
\item A linear form $L:\RX\to \R$ which is nonnegative on 
$\quadraticmodule{g}$ has a representing probability measure $\mu$ on $K$. 
\item A polynomial $p\in\RX$ which is strictly positive on $K$ 
is contained in $\quadraticmodule{g}$. 
\end{enumerate}
\end{theorem}
This motivates the definition of the lower bounds 
\begin{equation}
    f^*
\geq   f_{\mathrm{mom}} 
\coloneqq  \begin{array}[t]{rl} 
    \inf        &   L(f) \\ 
    \mbox{s.t.} &   L\in\RX^*,\,L(1)=1,\\ 
                &   \forall \,p\in\quadraticmodule{g}: L(p)\geq 0 
    \end{array} 
\geq   f_{\mathrm{sos}} 
\coloneqq  \begin{array}[t]{rl} 
    \sup        &   \lambda \\ 
    \mbox{s.t.} &   \lambda\in\R, \\ 
                &   f - \lambda \in 
                    \quadraticmodule{g} 
    \end{array}
\end{equation}
and, if $\quadraticmodule{g}$ is Archimedean, then \Cref{thm_putinar} implies $f_{\mathrm{mom}} = f_{\mathrm{sos}} = f^*$. 
Note that $\quadraticmodule{g}$ being Archimedean implies that 
the feasible region $K$ is compact. 
Conversely, if $K$ is compact, one may add 
$g_{\ell+1}=R^2-\var{1}^2-\ldots-\var{n}^2$ for a radius $R\in\R$ 
sufficiently large to the list of constraints 
in order to enforce the Archimedean property 
without changing the resulting semialgebraic set $K$. 

\subsection{Linear Representation Theory}

Throughout, let $\lineargroup \subseteq \mathrm{GL}_n(\R)$ be a finite matrix group. 
The following results on the representation theory of finite groups can be found in \cite{serre77,fulton13}. 
For further computational interests, we refer to \cite{FasslerStiefel92,Sturmfels08}. 
Although there are no novel contributions in this section, 
we still prove some of the presented statements in \Cref{appendix_rep_theo} using our own notation. 



The group $\lineargroup$ has an induced action on polynomials $p\in\RX$ by 
$p^\sigma(X)\coloneqq p(\sigma^{-1}\cdot X)$, 
where $\cdot$ denotes the matrix-vector multiplication.

\begin{definition}
A polynomial $p\in\RX$ is called \textbf{$\lineargroup$-invariant} 
if $p^\sigma=p$ for all $\sigma\in\lineargroup$. 
The set of all invariant polynomials is a 
$\R$-algebra denoted $\RX^\lineargroup$ and
the \textbf{Reynolds operator} is
\[
    \reynolds:
    \RX\to\RX^\lineargroup,\,
    p\mapsto \frac{1}{\vert\lineargroup\vert}
    \sum\limits_{\sigma\in\lineargroup} p^\sigma.
\]
\end{definition}

\begin{definition}\label{defi_rep_theo}
Let $\mathbb{K}$ be an arbitrary field. 
A \textbf{$\lineargroup$-module} $W$ is a $\mathbb{K}$-vector space together 
with a group homomorphism $\rho_W: \lineargroup\to\mathrm{GL}(W)$, 
called \textbf{representation}. 
A $\lineargroup$-module $W$ is called \textbf{irreducible} 
if its only $\lineargroup$-submodules are $0$ and $W$ itself. 
A \textbf{$\lineargroup$-module homomorphism} between two 
$\lineargroup$-modules $W,\tilde{W}$ is a $\mathbb{K}$-vector space homomorphism 
$\phi : W \to \tilde{W}$ with $\phi \circ \rho_W (\sigma) = \rho_{\tilde{W}} (\sigma) \circ \phi$ 
for all $\sigma \in \lineargroup$. 
The $\mathbb{K}$-vector space of $\lineargroup$-module homomorphisms 
is denoted by $\mathrm{Hom}_\lineargroup(W,\tilde{W})$. 
Two $\lineargroup$-modules are called \textbf{isomorphic} 
if there exists a bijective $\lineargroup$-module homomorphism between them. 
\end{definition}

Even though our original polynomials $f,g_1,\ldots,g_\ell$ for \ref{eq_pop} and 
the group $\lineargroup$ itself are defined over the real numbers $\R$, 
all representations will formally be defined 
over the complex numbers $\mathbb{K}=\C$, 
because we rely on Maschke's Theorem (\Cref{appendix_rep_theo}), 
which requires algebraically closed fields. 

The induced action extends naturally to $\CX$, 
making it a $\lineargroup$-module. 
Since the polynomial degree is preserved, 
homogeneous polynomials remain homogeneous under the action. 
By Maschke's Theorem, 
the polynomial ring over $\C$ has an \textbf{isotypic decomposition} 
\begin{equation}\label{eq_isotypic_decomposition}
    \RX \hookrightarrow \CX
=   \bigoplus\limits_{i=1}^{\irreps}
    W^{(i)}
    \quad \mbox{with} \quad
    W^{(i)}
=   \bigoplus\limits_{j\in J^{(i)}}
    W^{(i)}_j,
\end{equation}
where, for fixed $i$, the $W^{(i)}_j$ are pairwise isomorphic complex irreducible 
$\lineargroup$-modules of $\C$-dimension $\irrepdim{i} \in \mathbb{N}$ 
indexed by a set $J^{(i)}$, see \cite[Ch.2,~Thm.4]{serre77}. 
Moreover, $\irreps$ is the number of non-isomorphic complex irreducible representations of $\lineargroup$ and equals the number of conjugacy classes of $\lineargroup$. 


A vector space basis that admits the above 
decomposition can be constructed as follows, 
although this is by no means canonical. 
Only the parameters $\irreps$ and $\irrepdim{i}$ 
are intrinsic to the group $\lineargroup$. 

\begin{proposition}[Symmetry-Adapted Basis]\label{prop_symmetry_adapted_basis}\cite[Ch.~2,~Prop.~8]{serre77}
Let $\{\vartheta^{(i)}\,\vert\,1\leq i\leq \irreps\}$ 
with $\vartheta^{(i)}:\lineargroup\to \mathrm{GL}_{\irrepdim{i}}(\C)$
be a family of pairwise nonisomorphic 
irreducible matrix representations of $\lineargroup$ 
and fix some $1\leq i\leq \irreps$. 
For $1 \leq u \leq \irrepdim{i}$, define the operator 
\[
    \isotypicprojection{i}{u}: 
    \CX\to \CX, 
    p \mapsto 
    \frac{\irrepdim{i}}{\vert\lineargroup\vert} 
    \sum\limits_{\sigma\in\lineargroup} 
    \left(\vartheta^{(i)}(\sigma^{-1})\right)_{1,u}\,p^\sigma . 
\]
Let $\left\{w^{(i)}_{1,j}\,\vert\,j\in J^{(i)}\right\}$ 
be a homogeneous basis for the image of 
$\isotypicprojection{i}{1}$. 
Then the set 
\[
    \left\{  w^{(i)}_{1,j},\,
        w^{(i)}_{2,j}\coloneqq\isotypicprojection{i}{2}(w^{(i)}_{1,j}),\,
        w^{(i)}_{3,j}\coloneqq\isotypicprojection{i}{3}(w^{(i)}_{1,j}),\,
        \ldots,\,
        w^{(i)}_{\irrepdim{i},j}\coloneqq\isotypicprojection{i}{\irrepdim{i}}(w^{(i)}_{1,j}) \right\}
\]
is a homogeneous vector space basis for the 
irreducible $\lineargroup$-module $W^{(i)}_j$. 
\end{proposition} 

The submodules $W^{(i)}_j$ are complex in general, 
but can be decomposed into a real and imaginary part, 
see \cite[Eq.~2.9]{riener2013exploiting}. 
In particular, one can always choose real basis elements $w^{(i)}_{1,j} \in \RX$. 

\begin{definition}
A linear form $L\in\RX^*$ is called \textbf{$\lineargroup$-invariant} 
if $L(p^\sigma)=L(p)$ for all $p\in\RX,\,\sigma\in\lineargroup$. 
A bilinear form $\mathcal{L}:\RX\times \RX\to \R$ is called \textbf{$\lineargroup$-equivariant} if $\mathcal{L}(p^\sigma,q^\sigma)=\mathcal{L}(p,q)$ for all $p,q\in\RX,\,\sigma\in\lineargroup$. 
\end{definition}

Note that a $\lineargroup$-invariant linear form is 
uniquely determined by its values on $\RX^\lineargroup$. 
The vector space of $\lineargroup$-invariant linear forms 
is therefore isomorphic to $(\RX^\lineargroup)^*$. 
Using Schur's Lemma (\Cref{appendix_rep_theo}), one can show the following. 


\begin{lemma}
\label{lemma_bilinear_form_on_symm_basis}
Let $\mathcal{L}:\RX\times \RX\to \R$ be a 
$\lineargroup$-equivariant bilinear form as well as 
$1\leq i,i'\leq \irreps$ and $j,j'\in J^{(i)}$. 
There exist scalars $c^{(i)}_{j,j'}\in\R$ such that, 
for all $1\leq u,u' \leq \irrepdim{i}$, we have
\[
    \mathcal{L}(w^{(i)}_{u,j} , w^{(i')}_{u',j'})
=   \begin{cases}
    0,              & \mbox{ if } i\neq i'\\
    0,              & \mbox{ if } u\neq u'\\
    c^{(i)}_{j,j'}, & \mbox{ otherwise}\\
    \end{cases} , 
\]
where $w^{(i)}_{u,j}$ are defined as in \Cref{prop_symmetry_adapted_basis}. 
\end{lemma}

In particular, a $\lineargroup$-equivariant bilinear form 
is determined by its values on the blocks 
$W^{(i)}_1 \times W^{(i)}_1$ with $1\leq i\leq \irreps$. 
For $r\in \N$, denote by $\RX_r$ the set of polynomials of degree at most $r$. 

\begin{corollary}[Block Diagonalization]
\label{lemma_blockdiagonal}
Let $r\in\mathbb{N}$, $1\leq i\leq \irreps$ and 
denote by $1,2,3,\ldots,\gmultiplicity{i}{r}$ 
the indices $j \in J^{(i)}$ with 
$\deg(w^{(i)}_{1,j})\leq r$. 
In the ordered basis 
\[
    \bigcup\limits_{i=1}^{\irreps}
    \bigcup\limits_{j=1}^{\gmultiplicity{i}{r}}
    \{w^{(i)}_{u,j}\,\vert\,1\leq u\leq \irrepdim{i}\} ,
\]
the matrix of a group element $\sigma\in\lineargroup$ acting on $\RX_{r}$ has $\irreps$ blocks, 
each consisting of $\gmultiplicity{i}{r}$ further equal $\irrepdim{i}\times \irrepdim{i}$ blocks. 
In particular, we obtain a block diagonal matrix representation 
$\vartheta\coloneqq\gmultiplicity{1}{r}\vartheta^{(1)} \oplus\ldots \oplus \gmultiplicity{\irreps}{r}\vartheta^{(\irreps)}: 
\lineargroup\to\mathrm{GL}(\RX_{r})$ of dimension $\dim(\RX_{r})=\binom{n+r}{r}$, 
where $\gmultiplicity{i}{r}$'s are the multiplicities of irreducible representations. 

On the other hand, if the basis is reordered to
\[
    \bigcup\limits_{i=1}^{\irreps}
    \bigcup\limits_{u=1}^{\irrepdim{i}}
    \{w^{(i)}_{u,j}\,\vert\,1\leq j\leq \gmultiplicity{i}{r}\} ,
\]
the matrix of a $\lineargroup$-equivariant bilinear form 
$\mathcal{L}:\RX_{r} \times \RX_{r} \to \R$ has $\irreps$ blocks, 
each consisting of $\irrepdim{i}$ further equal 
$\gmultiplicity{i}{r} \times \gmultiplicity{i}{r}$ blocks. 
\end{corollary}

\subsection{Symmetry Reduction via Symmetry-Adapted Bases}
We assume that \ref{eq_pop} is $\lineargroup$-symmetric, 
that is, $f, g_1, \ldots, g_\ell\in\RX^\lineargroup$. 
Based on the symmetry exploiting variant of \Cref{thm_putinar}, we recall how to compute the minimum $f^*$ 
 efficiently. 
We call 
$\mathcal{S}^{(i)}\coloneqq\{w^{(i)}_{1,j}\,\vert\,j\in J^{(i)}\}$ 
a \textbf{symmetry-adapted basis} 
for the $i$-th isotypic component. 
A $\lineargroup$-invariant linear form 
$L: \RX^{\lineargroup}\to \R$ admits a 
$\lineargroup$-equivariant bilinear form
\begin{equation}\label{eq_bilinear_form}
    \mathcal{L}^\lineargroup_{L}: 
    \RX \times \RX \to \R, 
    (p,q) \mapsto L (\reynolds(p\,q)), 
\end{equation}
which is determined by its values on the 
$\mathcal{S}^{(i)}$. 
Furthermore, 
\begin{equation}\label{eq_localized_linear_form}
    g_k*L: \RX^\lineargroup \to \R,
    p\mapsto L(g_k\,p)
\end{equation}
is also $\lineargroup$-invariant and 
admits a $\lineargroup$-equivariant bilinear form 
$\mathcal{L}^\lineargroup_{g_k*L}:\RX \times \RX \to \R$. 

\begin{theorem}[Symmetric Positivstellensatz]
\label{thm_symm_positivstellensatz}
\cite[Thm.~3.5]{riener2013exploiting}
Assume that $\quadraticmodule{g}$ is Archimedean. 
\begin{enumerate}
\item A $\lineargroup$-invariant linear form $L:\RX^\lineargroup\to \R$ has a representing $\lineargroup$-invariant probability measure on $K$ if and only if $\mathcal{L}^\lineargroup_{g_k*L}$ is positive semidefinite on $\mathcal{S}^{(i)}$ for all $0\leq k\leq \ell,\,1\leq i\leq \irreps$. 
\item A $\lineargroup$-invariant polynomial $p\in\RX^\lineargroup$ 
which is strictly positive on $K$ 
can be written as
\begin{equation}
    p
=   \sum\limits_{i=1}^\irreps
    \sum\limits_{k=0}^\ell 
    g_k \, \reynolds (q^{(i)}_k) , 
\end{equation}
where $q^{(i)}_k\in\soscone{i}{}$ are sums of squares 
of elements in the vector space generated by $\mathcal{S}^{(i)}$. 
\end{enumerate}
\end{theorem}

Set $g_0\coloneqq1$, $d_k\coloneqq\lceil\deg(g_k)/2\rceil$ and
\[
    r_{\min}
\coloneqq  \max \{ \lceil\deg(f)/2\rceil, d_1, \ldots, d_\ell \} . 
\]
For $r\geq r_{\min}$, 
the restriction of $L$ to $\RX^\lineargroup_{2r}$ admits bilinear forms 
\begin{equation}
    \mathcal{L}^\lineargroup_{g_k*L}:\,
    \RX_{r-d_k}\times\RX_{r-d_k}\to\R,
\end{equation}
which are determined by their values on 
$\mathcal{S}^{(i)}_{r-d_k}\coloneqq\{w^{(i)}_{1,j}\,\vert\,1\leq j\leq m^{(i)}_{r-d_k}\}$. 
The \textbf{truncated localized symmetry-adapted moment matrices} 
are therefore the block diagonal matrices
\begin{equation}\label{eq_moment_matrix_blocks}
    \gmomentmatrix{g_k*L}{r-d_k}
\coloneqq  \begin{pmatrix}
    \boxed{\momentblock{g_k*L}{1}{r-d_k}} &&\\
    &\ddots&\\
    &&\boxed{\momentblock{g_k*L}{\irreps}{r-d_k}}
    \end{pmatrix}
\end{equation}
with entries  
\begin{equation}\label{eq_moment_matrix_entries}
    (\momentblock{g_k*L}{i}{r-d_k})_{j,j'}
\coloneqq  \mathcal{L}^\lineargroup_{g_k*L} ( w^{(i)}_{1,j},w^{(i)}_{1,j'} ) ,\quad
    1\leq j,j'\leq m^{(i)}_{r-d_k} . 
\end{equation}
The (primal) \textbf{symmetry-adapted moment relaxation of order $r$} is
\begin{equation}
    f^*
\geq f^{r}_{\mathrm{mom}}
\coloneqq  \begin{array}[t]{rl}
    \inf        &   L(f) \\
    \mbox{s.t.} &   L\in(\RX_{2r}^\lineargroup)^*, L(1)=1, \\ 
                &   \forall\,0\leq k\leq \ell,
                    1\leq i\leq \irreps:
                    \momentblock{g_k*L}{i}{r}\succeq 0
    \end{array}
\end{equation}
and the (dual) \textbf{symmetry-adapted SOS reinforcement of order $r$} is
\begin{equation}
    f^*
\geq f^{r}_{\mathrm{sos}}
\coloneqq  \begin{array}[t]{rl}
    \sup        &   \lambda \\
    \mbox{s.t.} &   \lambda \in \R, q^{(i)}_k \in \soscone{i}{r-d_k}, \\
                &   f-\lambda = 
                    \sum\limits_{k=0}^\ell
                    \sum\limits_{i=1}^{\irreps}
                    \reynolds(q^{(i)}_k\,g_k)  .
    \end{array}
\end{equation}

\begin{theorem}
\cite[Thm.~3.4]{riener2013exploiting}
For $r\geq r_{\min}$, 
the sequences $(f^{r}_{\mathrm{mom}})_{r\in\mathbb{N}}$ and 
$(f^{r}_{\mathrm{sos}})_{r\in\mathbb{N}}$ are monotonically nondecreasing 
with $f^{r}_{\mathrm{mom}}\geq f^{r}_{\mathrm{sos}}$. 
If the quadratic module $\quadraticmodule{g}$ is Archimedean, 
then $f^* = f_{\mathrm{mom}} = \lim_{r\to\infty} f^{r}_{\mathrm{mom}} = f_{\mathrm{sos}} = \lim_{r\to\infty} f^{r}_{\mathrm{sos}}$. 
\end{theorem}


After reordering, we may assume without loss of generality 
that the first irreducible $\lineargroup$-modules 
$W^{(1)}_j$ in the isotypic decomposition 
afford the trivial representation with $\vartheta^{(1)}(s) = 1$ 
for every $s\in\lineargroup$ of dimension $\irrepdim{1}=1$. 
In particular, 
\begin{equation}
    \reynolds
=   \isotypicprojection{1}{1}:
    \RX\to \RX, 
    p\mapsto 
    \frac{1}{\vert\lineargroup\vert}
    \sum\limits_{\sigma\in\lineargroup}
    p^\sigma 
\end{equation}
is the Reynolds operator with image $\RX^\lineargroup$ and the set 
\begin{equation}
    \mathcal{S}^{(1)}_{2r}
=   \{w^{(1)}_{1,1},\ldots,w^{(1)}_{1,\gmultiplicity{1}{2r}}\}
\end{equation}
is a basis for the vector space $\RX^\lineargroup_{2r}$ of 
$\lineargroup$-invariant polynomials of degree at most $2r$. 
The \textbf{pseudomoments} of a $\lineargroup$-invariant linear form $L$ 
in this basis are $\gmoment{j} \coloneqq L(w^{(1)}_{1,j})$ and 
form a sequence $\mathbf{y}\in\R^{\gmultiplicity{1}{2r}}$. 
Moreover, we assume that 
$w^{(1)}_{1,1}=1$. 
Hence, 
if $L$ has a representing probability measure, 
then $\gmoment{1} = L(1) = 1$. 
Denote by $\sdpcone{i}{r}$ the vector space of real symmetric 
$\gmultiplicity{i}{r} \times \gmultiplicity{i}{r}$ matrices. 
There exist unique scalar coefficients $f_j\in\R$ and 
matrix coefficients $\sdpmatrix{i}{r,k,j} \in \sdpcone{i}{r-d_k}$, 
such that 
\begin{equation}\label{eq_sdp_coefficients}
    f
=   \sum\limits_{j=1}^{\gmultiplicity{1}{2r}} f_j\,w^{(1)}_{1,j}
    \quad \mbox{and} \quad
    \reynolds
    (w^{(i)}_{1,j'} \, w^{(i)}_{1,j''} \, g_k)
=   \sum\limits_{j=1}^{\gmultiplicity{1}{2r}} 
    (\sdpmatrix{i}{r,k,j})_{j',j''}\,w^{(1)}_{1,j} .
\end{equation}
In particular, we have 
\begin{equation}
    L(f)
=   \sum\limits_{j=1}^{\gmultiplicity{1}{2r}} f_j\,\gmoment{j}
    \quad \mbox{and} \quad
    \momentblock{g_k*L}{i}{r}
=   \sum\limits_{j=1}^{\gmultiplicity{1}{2r}} 
    \sdpmatrix{i}{r,k,j}\,\gmoment{j} ,\, 
    0\leq k\leq \ell. 
\end{equation}
We also write $\momentblock{g_k*\mathbf{y}}{i}{r}$ 
instead of $\momentblock{g_k*L}{i}{r}$.

\begin{corollary}
\label{thm_symm_moment_sos_hierarchy_SDP}
The symmetry-adapted moment-SOS hierarchy of order $r$ 
can be written as an SDP:
\[
    f^{r}_{\mathrm{mom}}
=   \begin{array}[t]{rl}
    \inf        &   \sum_j f_j\,\gmoment{j} \\
    \mbox{s.t.} &   \gmoment{1} = 1,
                    \forall \, k,i: \\
                &   \sum_j \sdpmatrix{i}{r,k,j}\,\gmoment{j}
                    \succeq 0
    \end{array}
    \quad \mbox{and} \quad
    f^{r}_{\mathrm{sos}}
=   \begin{array}[t]{rl}
    \sup        &   f_1-
                    \sum_{k,i}
                    \trace(\sdpmatrix{i}{r,k,1}\,\sosmatrix{i}{k}) \\
    \mbox{s.t.} &   \sosmatrix{i}{k}\succeq 0,\forall\,j\geq 2: \\
                &   f_j=
                    \sum_{k,i}
                    \trace(\sdpmatrix{i}{r,k,j}\,\sosmatrix{i}{k}) ,
    \end{array}
\]
where the infimum is taken over finite sequences 
$\mathbf{y}\in\R^{\gmultiplicity{1}{2r}}$ and
the supremum is taken over positive semidefinite matrices 
$\sosmatrix{i}{k} \in \sdpcone{i}{r-d_k}$. 
\end{corollary}

\begin{remark}
If $\gmomentsequence,\sosmatrix{i}{k}$ 
are optimal for the SDP, 
then the duality gap is 
\[
    f^{r}_{\mathrm{mom}} - f^{r}_{\mathrm{sos}} 
=   \sum_{k,i} \trace(
    \momentblock{g_k*\mathbf{y}}{i}{r-d_k} \cdot \sosmatrix{i}{k}) 
\geq 0.
\]
\end{remark}

%

\section{Symmetry Exploitation for Sparse Problems}
\label{section_3_symmetry_and_sparsity}

We now explain how to apply term sparsity exploitation when the moment-SOS-hierarchy is formulated in a symmetry-adapted basis. 
This leads to a generalization of the TSSOS hierarchy from \cite{magron21} and we prove convergence in \Cref{thm_sparsity_convergence_1,thm_sparsity_convergence_2}. 

\subsection{Chordal Graphs and Sparse Matrices}

Let $\mathrm{G}$ be an undirected graph with vertices $V$ and edges $E$. 
Then $\mathrm{G}$ is called a \textbf{chordal graph} if every cycle of $\mathrm{G}$ has a chord, that is, 
an edge that joins two nonconsecutive vertices. 
A \textbf{chordal extension} of $\mathrm{G}$ is a chordal graph 
that contains $\mathrm{G}$ as a subgraph. Chordal extensions are not unique. A \textbf{maximal} chordal extension is the chordal extension that completes every connected component of $G$.
A \textbf{clique} of $\mathrm{G}$ is a subset of $V$ that admits a complete subgraph, 
complete meaning that any two vertices are connected. 
A \textbf{maximal clique} of $\mathrm{G}$ is a clique of $\mathrm{G}$ that is not contained in another one. A chordal extension with the smallest clique number (i.e. the maximal size of maximal cliques) is called a \textbf{smallest} chordal extension.

\begin{definition}
A \textbf{binary matrix} $\mathbf{B}$ is a matrix with entries $0$ and $1$. 
\begin{itemize}
\item The support $\supp(\mathbf{B})$ is the set of all $(j,j')$ with $\mathbf{B}_{j,j'}=1$. 
\item If $\mathbf{B}$ is square, 
$\mathrm{Sym}(\mathbf{B})$ denotes the vector space of 
real symmetric matrices $\mathbf{Q}$ of the same size 
with $\mathbf{Q}_{j,j'}=0$ if $\mathbf{B}_{j,j'}=0$. 

\item The \textbf{block closure} of $\mathbf{B}$ is a 
binary matrix $\overline{\mathbf{B}}$ of the same size, 
which is minimal with respect to the properties
\begin{enumerate}
\item $\supp(\mathbf{B}) \subseteq \supp(\overline{\mathbf{B}})$ and
\item if $\overline{\mathbf{B}}_{j,j'}=\overline{\mathbf{B}}_{j',j''}=1$, 
then $\overline{\mathbf{B}}_{j,j''}=1$. 
\end{enumerate}
\end{itemize}
\end{definition}

\begin{remark}
A symmetric binary matrix is the companion matrix of an undirected graph, and the block closure is the matrix of the maximal chordal extension (obtained after adding the dotted edge in the graph below). 
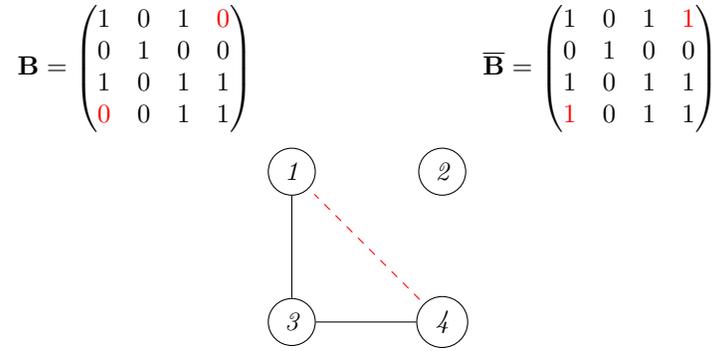
\[
	\mathbf{B}
=   \begin{pmatrix}
	1 & 0 & 1 & \textcolor{red}{0} \\
	0 & 1 & 0 & 0 \\
	1 & 0 & 1 & 1 \\
	\textcolor{red}{0} & 0 & 1 & 1 \\
	\end{pmatrix}
	\hspace{3cm}
	\overline{\mathbf{B}}
=   \begin{pmatrix}
	1 & 0 & 1 & \textcolor{red}{1} \\
	0 & 1 & 0 & 0 \\
	1 & 0 & 1 & 1 \\
	\textcolor{red}{1} & 0 & 1 & 1 \\
	\end{pmatrix}
\]
\[
		\begin{tikzpicture}
			\node[shape=circle,draw=black] (1) at (0,2) {1};
			\node[shape=circle,draw=black] (3) at (0,0) {3};
			\node[shape=circle,draw=black] (4) at (2,0) {4};
			\node[shape=circle,draw=black] (2) at (2,2) {2};
			
			\path [-] (1) edge node[left] {} (3);
			\path [-] (3) edge node[left] {} (4);
			
			\draw (1.7,0.3) -- (0.3,1.7)[dashed,color=red];
		\end{tikzpicture}
		\]
\end{remark}

\subsection{Term Sparsity}



Let $f,g_1,\ldots,g_\ell\in\RX^\lineargroup$ be as in \Cref{eq_pop} and 
$w^{(i)}_{1,j}$ be a symmetry-adapted basis as in \Cref{lemma_blockdiagonal}. 

\begin{definition}\label{defi_tsp_matrix}
For $p=\sum_j p_j\,w^{(1)}_{1,j} \in \RX^\lineargroup_{2r}$, 
denote by $\supp(p)$ the set of all 
$1\leq j\leq \gmultiplicity{1}{2r}$ with $p_j\neq 0$, and set
\[
    \mathcal{B}_r
\coloneqq  \supp(f) \cup 
    \bigcup\limits_{k=1}^\ell 
    \supp(g_k) \cup 
    \bigcup\limits_{i=1}^{\irreps}
    \bigcup\limits_{j=1}^{\gmultiplicity{1}{r}}
    \supp\left(\reynolds((w^{(i)}_{1,j})^2)\right)
\]
as well as $\tspsupport{i}{r,0,0} \coloneqq \mathcal{B}_r$ and 
$\tspsupport{i}{r,0,k} \coloneqq \emptyset$. 
For $0\leq k\leq \ell, 1\leq i\leq \irreps$ and $s\geq 1$, 
we define a symmetric binary matrix 
$\tspmatrix{i}{r,s,k} \in \sdpcone{i}{r-d_k}$ 
in two successive steps:
\begin{enumerate}
\item \underline{Support extension}: 
Define a binary matrix 
$\pretspmatrix{i}{r,s,k} \in \sdpcone{i}{r-d_k}$ 
with entries
\[
    (\pretspmatrix{i}{r,s,k})_{j,j'}
\coloneqq  \begin{cases}
    1,  &   \text{if } \supp\left(
            \reynolds(
            w^{(i)}_{1,j}\,w^{(i)}_{1,j'}\,g_k)
            \right) \cap
            \bigcup\limits_{k'=0}^\ell
            \tspsupport{i}{r,s-1,k'} 
            \neq \emptyset, \\
    0,  &   \mbox{otherwise}.
    \end{cases}
\]
The corresponding graph $\mathrm{G}^{\mathrm{tsp}}$ 
is called the \textbf{tsp graph of sparsity order $s$}. 
\item \underline{Block closure}: 
Choose
$\tspmatrix{i}{r,s,k}\coloneqq\overline{\pretspmatrix{i}{r,s,k}}$
and set
\[
    \tspsupport{i}{r,s,k}
\coloneqq  \bigcup\limits_{(\pretspmatrix{i}{r,s,k})_{j,j'}=1}
    \supp\left(
    \reynolds(w^{(i)}_{1,j}\,w^{(i)}_{1,j'}\,g_k)
    \right) .
\]
\end{enumerate} 
\end{definition}

Before we build a hierarchy around this definition, we emphasize the following. 

\begin{remark}
Ignoring the ``diagonal squares'', that is, the diagonal elements of the moment matrix, 
\[
    \bigcup\limits_{i=1}^{\irreps}
    \bigcup\limits_{j=1}^{\gmultiplicity{1}{r}}
    \supp\left(\reynolds((w^{(i)}_{1,j})^2)\right)
\]
in the definition of $\mathcal{B}_r$ is possible and can 
make the computations more efficient (\Cref{fig_1D_ring,fig_2D_torus}). 
Similarly, the choice of block closures/chordal extensions in the definition of 
$\tspmatrix{i}{r,s,k}$ influences the convergence; 
see \Cref{thm_sparsity_convergence_1,thm_sparsity_convergence_2} and \Cref{fig_1D_ring,fig_2D_torus,fig_daniel_vs_pop}. 

The initial TSSOS method \cite{magron21} defines the tsp graph 
in the standard monomial basis $\{x^\alpha\}_{\vert\alpha\vert\leq 2r}$ 
instead of the symmetry-adapted basis $\{w^{(i)}_{1,j}\}$. 
In particular, there is no need for a Reynolds operator, but sign symmetry with respect to the two-element group $\{\pm 1\}$ could already be exploited earlier. 
We give a quantitative comparison in \Cref{section_4_benchmarks}. 
\end{remark}

\subsection{Example}

We construct the tsp graph for the example in \cite[Ex.~4.8]{riener2013exploiting} for the symmetric group 
$\lineargroup = \symmetricgroup{3}$ of order $6$. 
It is generated by the transpositions 
$\sigma_1=(1,2)$ and $\sigma_2=(2,3)$, 
or, as a matrix group, by 
\begin{equation}
    \sigma_1 
=   \begin{pmatrix} 
    0 & 1 & 0 \\
    1 & 0 & 0 \\
    0 & 0 & 1
    \end{pmatrix}
    \quad \mbox{and} \quad 
    \sigma_2
=   \begin{pmatrix} 
    1 & 0 & 0 \\
    0 & 0 & 1 \\
    0 & 1 & 0
    \end{pmatrix} .
\end{equation}
There are $h=3$ conjugacy classes, represented by 
$\mathrm{id}$, $\sigma_1$, $\sigma_1\,\sigma_2$, 
and thus $3$ complex irreducible representations, 
encoded by partitions $\lambda$ of $n=3$ \cite[Ch.4.2]{fulton13}. 
The corresponding $\lineargroup$-modules are the Specht modules $\mathbb{S}^\lambda$ \cite[Ch.7.1]{JamesKerber1984}. 
We have
the trivial representation 
$\vartheta^{(1)}:\symmetricgroup{3} \to \mathrm{GL}(\mathbb{S}^{3})$ 
of dimension $1$, 
the reflection representation 
$\vartheta^{(2)}:\symmetricgroup{3} \to \mathrm{GL}(\mathbb{S}^{21})$ 
of dimension $2$ and 
the sign representation 
$\vartheta^{(3)}:\symmetricgroup{3} \to \mathrm{GL}(\mathbb{S}^{111})$ 
of dimension $1$. 
The resulting \textbf{character table} contains the traces of these 
representations on the conjugacy classes 
\begin{equation}
	\chi
=   (\chi^{(i)}(\sigma))_{i,\sigma}
=   (\mathrm{Trace}(\vartheta^{(i)}(\sigma)))_{i,\sigma}
=	\begin{matrix}
        &   
        \begin{matrix}
            \phantom{-} \textcolor{gray}{\mathrm{id}} & 
            \phantom{..} \textcolor{gray}{\sigma_1} \phantom{..} & 
            \textcolor{gray}{\sigma_1\,\sigma_2} 
        \end{matrix} \\
	    \begin{matrix}
            \phantom{-} \textcolor{gray}{\mathbb{S}^{3}} \\ 
            \phantom{-} \textcolor{gray}{\mathbb{S}^{21}} \\ 
            \phantom{-} \textcolor{gray}{\mathbb{S}^{111}} 
        \end{matrix} 
        & 
	    \begin{pmatrix} 
            \phantom{-} 1 \phantom{.}   &		   -  1 \phantom{-}	    &  \phantom{-} 1    &\\ 
		      \phantom{-} 2 \phantom{.}	  &	\phantom{-} 0 \phantom{-}     &   		  -  1	  &\\ 
		      \phantom{-} 1 \phantom{.}	  &	\phantom{-} 1 \phantom{-}	  &  \phantom{-} 1 	  &
        \end{pmatrix}
	\end{matrix} . 
\end{equation}
The induced action on $\RX=\R[x_1,x_2,x_3]$ is given by permutation of variables. 
Let us consider polynomials up to degree $r=2$. 
This defines a representation 
$\vartheta: \lineargroup \to \mathrm{GL}(\RX_{2})$ 
of dimension $\binom{3+2}{2} = 10$. 
The multiplicities of the irreducible representations in $\vartheta$ 
are computed by solving the $\irreps$-dimensional linear system 
\begin{equation}
    \trace(\vartheta(\sigma)) 
=   m^{(1)}_2 \, \chi^{(1)}(\sigma) + 
    m^{(2)}_2 \, \chi^{(2)}(\sigma) + 
    m^{(3)}_2 \, \chi^{(3)}(\sigma) 
\end{equation}
on the conjugacy classes, 
which yields 
$\gmultiplicity{1}{2}=4,\,\gmultiplicity{2}{2}=3,\,\gmultiplicity{3}{2}=0$. 
A symmetry-adapted basis for $\RX_{2}$ consists of 
\begin{equation}
    \mathcal{S}^{(1)}_2
=   \{  1,\,
        a_1=x_1+x_2+x_3,\,
        a_2=x_1^2+x_2^2+x_3^2,\,
        a_{11}=x_1 x_2+x_2 x_3+x_3 x_1\}
\end{equation}
and
\begin{equation}
    \mathcal{S}^{(2)}_2
=   \{  b_1=2 x_3-x_2-x_1,\,
        b_2=2 x_3^2-x_2^2-x_1^2,\,
        b_{11}=-2 x_1 x_2+x_2 x_3+x_3 x_1\} . 
\end{equation}
There is no $\mathcal{S}^{(3)}_2$, because $\gmultiplicity{3}{2}=0$. 
Hence, the symmetry-adapted moment matrix of a $\symmetricgroup{3}$-invariant 
linear form $L$ with moments $\gmomentsequence$ has two blocks, namely
\begin{equation}
    \momentblock{\gmomentsequence}{1}{2}
=   \begin{matrix}
    &\textcolor{gray}{
    \begin{matrix}
    	1\phantom{11111} & a_1\phantom{11111111} & a_2\phantom{1111111} & a_{11}\phantom{111}
    \end{matrix}
    }\\
	\begin{matrix}  \textcolor{gray}{1} \\ \textcolor{gray}{a_1} \\ \textcolor{gray}{a_2} \\ \textcolor{gray}{a_{11}} \end{matrix} & 
	\begin{pmatrix} 
    1 & 
    \gmoment{1}  & 
    \gmoment{2}  & 
    \gmoment{11} \\
    \gmoment{1}  & 
    \gmoment{2} + 2 \gmoment{11} & 
    \gmoment{3} + 2 \gmoment{21} & 
    2 \, \gmoment{21} + \gmoment{111} \\
    \gmoment{2}  & 
    \gmoment{3} + 2 \gmoment{21} & 
    \gmoment{4} + 2 \gmoment{22} & 
    2 \, \gmoment{31} + \gmoment{211} \\
    \gmoment{11} & 
    2 \, \gmoment{21} + \gmoment{111} & 
    2 \, \gmoment{31} + \gmoment{211} & 
    \gmoment{22} + 2 \gmoment{211} \\
    \end{pmatrix}
	\end{matrix}
\end{equation}
and
\begin{equation}
    \momentblock{\gmomentsequence}{2}{2}
=   \begin{matrix}
	&\textcolor{gray}{
		\begin{matrix}
			b_1\phantom{111111111} & b_2\phantom{11111111} & b_{11}\phantom{1}\\
		\end{matrix}
	}\\
	\begin{matrix}  \textcolor{gray}{b_1} \\ \textcolor{gray}{b_1} \\ \textcolor{gray}{b_{11}} \end{matrix} & 
	\begin{pmatrix} 
    2\gmoment{2}-2\gmoment{11} & 2\gmoment{3}-2\gmoment{21} & 2\gmoment{21}-2\gmoment{111} \\
    2\gmoment{3}-2\gmoment{21} & 2\gmoment{4}-2\gmoment{22} & 2\gmoment{31}-2\gmoment{211} \\
    2\gmoment{21}-2\gmoment{111} & 2\gmoment{31}-2\gmoment{211} & 2\gmoment{22}-2\gmoment{211} \\
    \end{pmatrix}
	\end{matrix} .
\end{equation}
For 
\begin{equation}
    f
=   1+x_1+x_2+x_3+(x_1+x_2+x_3)^2+x_1^4+x_2^4+x_3^4 
=   1+a_1+a_2+2\,a_{11}+a_4 
\end{equation}
we obtain
\begin{align}
    \mathcal{B}_2
&=  \supp(f) \cup 
    \bigcup\limits_{w\in\mathcal{S}^{(1)}_2}
    \supp\left(\reynolds(w^2)\right) \cup 
    \bigcup\limits_{w\in\mathcal{S}^{(2)}_2}
    \supp\left(\reynolds(w^2)\right) \\ 
&=  \{ 1, a_1, a_2, a_{11}, a_4\} \cup 
    \{ 1, a_2, a_{11}, a_4, a_{22}, a_{211} \} \cup 
    \{ a_2, a_{11}, a_4, a_{22}, a_{211} \} 
\end{align}
and set $\tspsupport{1}{2,0}=\tspsupport{2}{2,0}=\mathcal{B}$ 
(we omit the index $k$ since there are no constraints). 
In the tsp graph for the sparsity order $s=1$, the nodes $a_1$ and $a_{11}$ are disconnected 
because the support of $\reynolds(a_1\,a_{11}) = 2\,a_{21} + a_{111}$ has an empty intersection with $\tspsupport{1}{r,0}$. 
But since $a_1$ is connected to $a_2$ and $a_2$ is connected to $a_{11}$, 
$a_1$ and $a_{11}$ are connected in the chordal extension. 
Similarly, $b_1$ is connected neither to $b_2$ nor to $b_{11}$. 
Hence, the companion matrices of the tsp graphs for $s=1$ are the symmetric binary matrices
\begin{equation}
    \pretspmatrix{1}{2,1}
=   \begin{pmatrix}
    1 & 1 & 1 & 1 \\
    1 & 1 & 1 & \textcolor{red}{0} \\
    1 & 1 & 1 & 1 \\
    1 & \textcolor{red}{0} & 1 & 1 \\
    \end{pmatrix}
    \quad \mbox{with block closure} \quad
    \tspmatrix{1}{2,1}
=   \begin{pmatrix}
    1 & 1 & 1 & 1 \\
    1 & 1 & 1 & \textcolor{red}{1} \\
    1 & 1 & 1 & 1 \\
    1 & \textcolor{red}{1} & 1 & 1 \\
    \end{pmatrix}
\end{equation}
and
\begin{equation}
    \pretspmatrix{2}{2,1}
=   \tspmatrix{2}{r,1}
=   \begin{pmatrix}
    1 & 0 & 0 \\
    0 & 1 & 1 \\
    0 & 1 & 1 \\
    \end{pmatrix} . 
\end{equation}
We set
\begin{equation}
    \tspsupport{2}{2,1} 
=   \bigcup\limits_{(\pretspmatrix{2}{r,1})_{v,w}=1} 
    \supp \left( \reynolds(v\,w) \right) 
=   \{ a_{2}, a_{11}, a_4, a_{31}, a_{22}, a_{211} \} . 
\end{equation}
The new element $a_{31}$ does not lead to new entries 
in $\tspmatrix{2}{2}$ and hence we have 
\begin{equation}
    \tspmatrix{1}{2,1}
=   \tspmatrix{1}{2,s}
    \quad \mbox{and} \quad
    \tspmatrix{2}{2,1}
=   \tspmatrix{2}{2,s}
\end{equation}
for all $s\geq 1$. 

\subsection{Generalized TSSOS Hierarchy and Asymptotic Convergence}

The \textbf{symmetry-adapted TSSOS hierarchy for POP of order $r,s$} is 
\begin{equation}\label{eq_symmetric_tssos_hiersarchy}
    f^{r,s}_{\mathrm{mom}}
=   \begin{array}[t]{rl}
    \inf        &   \sum\limits_j f_j\,\gmoment{j} \\
    \mbox{s.t.} &   \gmomentsequence\in\R^{\gmultiplicity{1}{2r}},\,
                    \gmoment{1} = 1,\\
                &   j\notin \bigcup\limits_{k,i} \tspsupport{i}{r,s,k}
                    \Rightarrow \gmoment{j}=0,\\
                &   \forall \, k,i: \, 
                    \sum\limits_j \sdpmatrix{i}{r,s,k,j}\,\gmoment{j}
                    \succeq 0
    \end{array}
    \quad \mbox{and} \quad
    f^{r,s}_{\mathrm{sos}}
=   \begin{array}[t]{rl}
    \sup        &   f_1-
                    \sum\limits_{k,i}
                    \trace(\sdpmatrix{i}{r,s,k,1}\cdot\sosmatrix{i}{k}) \\
    \mbox{s.t.} &   \sosmatrix{i}{k}\in\sdpcone{i}{r-d_k}
                    (\tspmatrix{i}{r,s,k}),\\
                &   \sosmatrix{i}{k}\succeq 0,
                    \forall\,j \in \bigcup\limits_{k,i} \tspsupport{i}{r,s,k}: \\
                &   f_j=
                    \sum\limits_{k,i}
                    \trace(\sdpmatrix{i}{r,s,k,j}\cdot\sosmatrix{i}{k}) ,
    \end{array}
\end{equation}
where $\sdpmatrix{i}{r,s,k,j}$ is the Hadamard product 
$\tspmatrix{i}{r,s,k} \circ \sdpmatrix{i}{r,k,j}$. 
By construction, it follows that 
$\supp(\pretspmatrix{i}{r,s,k}) \subseteq 
\supp(\pretspmatrix{i}{r,s+1,k})$ and hence
the sequence of binary matrices 
$(\tspmatrix{i}{r,s,k})_{s\geq 1}$ stabilizes. 
The stabilized matrices are denoted by 
$\tspmatrix{i}{r,*,k}$. 



\begin{theorem}\label{thm_sparsity_convergence_1}
Assume that the block closure in \Cref{defi_tsp_matrix} is maximal. 
For fixed $r\geq r_{\min}$, the sequences 
$(f^{r,s}_{\mathrm{mom}})_{s\geq 1}$ and 
$(f^{r,s}_{\mathrm{sos}})_{s\geq 1}$ are 
monotonously nondecreasing with 
$f^{r,*}_{\mathrm{mom}}=f^r_{\mathrm{mom}}$ and 
$f^{r,*}_{\mathrm{sos}}=f^r_{\mathrm{sos}}$. 
\end{theorem}
\begin{proof}
Monotonicity and $f^{r,*}_{\mathrm{mom}} \leq f^r_{\mathrm{mom}}$ are clear from the construction. 
Let
\[
    \mathcal{B}_{r,*}
\coloneqq  \bigcup\limits_{k=0}^\ell
    \bigcup\limits_{i=1}^\irreps
    \bigcup\limits_{(\tspmatrix{i}{r,*,k})_{j,j'}=1}
    \supp\left(\reynolds(
    w^{(i)}_{1,j}\,w^{(i)}_{1,j'}\,g_k)\right).
\]
Assume that $\gmomentsequence=(\gmoment{j})_{j\in\mathcal{B}_{r,*}}$ 
is feasible for $f^{r,*}_{\mathrm{mom}}$ and set
\[
    \mathbf{\overline{y}}_j
\coloneqq  \begin{cases}
    \gmoment{j},    &   j\in\mathcal{B}_{r,*} , \\
    0,              &   j\in\{1,\ldots,\gmultiplicity{1}{2r}\}
                        \setminus\mathcal{B}_{r,*} .
    \end{cases}
\]
Since $\tspmatrix{i}{*,k}$ is stable under maximal block closure, 
we have 
\[
    \supp\left(\reynolds(
    w^{(i)}_{1,j}\,w^{(i)}_{1,j'}\,g_k)\right) 
    \cap \mathcal{B}_{r,*}
=   \emptyset
\]
whenever $(\tspmatrix{i}{r,*,k})_{j,j'}=0$ and thus 
$\momentblock{g_k*\overline{\gmomentsequence}}{i}{r} = 
\tspmatrix{i}{r,*,k} \circ \momentblock{g_k*\gmomentsequence}{i}{r}$. 
Hence, $\mathbf{\overline{y}}$ is feasible for $f^r_{\mathrm{mom}}$ with 
$\sum_j f_j\,\gmoment{j} = \sum_j f_j\,\mathbf{\overline{y}}_j \geq f^r_{\mathrm{mom}}$. 
Since $\gmomentsequence$ was arbitrary, we also have $f^{r,*}_{\mathrm{mom}} \geq f^r_{\mathrm{mom}}$. 
Dualizing yields the same for $f^{r,*}_{\mathrm{sos}}$. 
\end{proof}

\begin{theorem}\label{thm_sparsity_convergence_2}
Assume that the block closure in \Cref{defi_tsp_matrix} is maximal. 
For fixed sparsity order $s\geq 1$, the sequences 
$(f^{r,s}_{\mathrm{mom}})_{r\geq r_{\min}}$ and 
$(f^{r,s}_{\mathrm{sos}})_{r\geq r_{\min}}$ are 
monotonously nondecreasing. 
\end{theorem}
\begin{proof}
Denote the tsp matrix for the order $r$ by $\tspmatrix{i}{r,s,k}$ 
and for $r+1$ by $\mathbf{C}^{(i)}_{r,s,k}$. 
We show that $(\mathbf{C}^{(i)}_{r,s,k})_{j,j'}=1$ implies $(\tspmatrix{i}{r,s,k})_{j,j'}=1$ by induction on $s$. 
By construction $(\mathbf{\tilde{C}}^{(i)}_{r,1,k})_{j,j'}=1$ implies 
$(\pretspmatrix{i}{r,1,k})_{j,j'}=1$ and thus 
$(\mathbf{C}^{(i)}_{r,1,k})_{j,j'}=1$ implies $(\mathbf{B}^{(i)}_{r,1,k})_{j,j'}=1$ for all 
$1\leq i\leq \irreps,\,1\leq k\leq \ell$. 
Assume that $(\mathbf{C}^{(i)}_{r,s,k})_{j,j'}=1$ implies 
$(\tspmatrix{i}{r,s,k})_{j,j'}=1$ for some $s\geq 1$. 
Then $\tspsupport{i}{r,s,k} \subseteq \mathcal{C}^{(i)}_{r,s,k}$ 
and thus, by construction, 
$(\mathbf{{C}}^{(i)}_{r,s+1,k})_{j,j'}=1$ implies 
$(\tspmatrix{i}{r,s+1,k})_{j,j'}=1$. 
\end{proof}


\begin{remark}
Assume that $K$ has nonempty interior. 
For fixed $r\geq r_{\min}$ and $s\geq 1$, 
there is no duality gap, that is, 
$f^{r,s}_{\mathrm{mom}} - f^{r,s}_{\mathrm{sos}} = 0$. 
This can be proven analogously to \cite[Thm. 4.2]{lasserre01}. 
Similarly, if there is an explicit ball constraint among the $g_k$, 
then the duality gap also vanishes; see \cite{josz2016strong}. 
\end{remark}


\begin{example}
The Robinson polynomial 
\[
    f 
\coloneqq  x_1^6 + x_2^6 - x_1^4\,x_2^2 - x_2^4\,x_1^2 - x_1^4 - x_2^4 - x_1^2 - x_2^2 + 3\,x_1^2\,x_2^2 + 1
\]
is invariant under the $\mathfrak{S}_2$-permutation of $x_1$ and $x_2$, 
and the corresponding symmetry-adapted basis consists of 
\[
    \mathcal{S}^{(1)}_3
=   \{1,\,x_1\,x_2,x_1+x_2,x_1^2+x_2^2,x_1^3+x_2^3,x_1\,x_2^2+x_1^2\,x_2\}
\]
and
\[
    \mathcal{S}^{(2)}_3
=   \{x_1-x_2,x_1^2-x_2^2,x_1^3-x_2^3,x_1\,x_2^2-x_1^2\,x_2\}. 
\]
The first term sparsity iteration ($s=1$) gives the block structure 
\[
    \tspsupport{1}{3,1}
=   \{1,\,x_1^2+x_2^2\}\cup\{x_1\,x_2\}\cup\{x_1+x_2,\,x_1^3+x_2^3,\,x_1\,x_2^2+x_1^2\,x_2\}
\]
and
\[
    \tspsupport{2}{3,1}
=   \{x_1-x_2,\,x_1^3-x_2^3,\,x_1\,x_2^2-x_1^2\,x_2\}\cup\{x_1^2-x_2^2\}
\]
with a lower bound of $-0.9338$. 
The second iteration ($s=2$) gives the block structure 
\[
    \tspsupport{1}{3,2}
=   \{1,\,x_1\,x_2,\,x_1^2+x_2^2\}\cup\{x_1+x_2,\,x_1^3+x_2^3,\,x_1\,x_2^2+x_1^2\,x_2\}
\]
and
\[
    \tspsupport{2}{3,2}
=   \{x_1-x_2,\,x_1^3-x_2^3,\,x_1\,x_2^2-x_1^2\,x_2\}\cup\{x_1^2-x_2^2\}
\]
with the (numerically) same lower bound of $-0.9338$. 
Note that the block structure of the first term sparsity iteration is also compatible with the sign symmetry of $f$. 
\end{example}




\section{Benchmarks}
\label{section_4_benchmarks}

In this section, we compare with existing techniques. 
We consider three classes of examples, quartics with different symmetries and constraints. 
Polynomial optimization problems are modeled using the Julia \cite{bezanson2017julia} package {\tt TSSOS} \cite{tssossoft} (see also \cite[Appendix B]{magron2023sparse}) and the solver {\tt MOSEK} \cite{andersen2000MOSEK} was used to solve SDP relaxations. 
We rely on the {Julia} package \href{https://github.com/kalmarek/SymbolicWedderburn.jl?tab=readme-ov-file}{\tt SymbolicWedderburn.jl} to compute symmetry-adapted bases; see also \cite{Kaluba19}. 
All experiments were performed on a computer with an 
AMD Ryzen AI 7 PRO 360 w/ Radeon 880M CPU @2.00 GHz, 
8 cores, 16 threads, and 64GB of RAM. 
All results can be reproduced with online scripts\footnote{\url{https://github.com/wangjie212/TSSOS/tree/master/example}}. 

\subsection{1-D Ring Ising Quartic with Dihedral Symmetry}

The dihedral group $\mathfrak{D}_{2n}$ of order $2n$ 
is generated by a cycle $\rho$ of length $n$ and a reflection $\sigma$, 
whose action on the indeterminates $\var{i},\,1\leq i\leq n,$ is defined by
\begin{equation}
    \rho(\var{i}) = \var{i+1}
    \quad \mbox{and} \quad
    \sigma(\var{i}) = \var{-i} 
    \quad (\mbox{indices mod } n). 
\end{equation}
For parameters $a,b,c,d\in\R$, consider the quartic 
\begin{equation}
    f
\coloneqq  \sum_{i=1}^n 
    a\,\var{i}^4 + b\,\var{i}^2\,\var{i+1}^2 + c\,\var{i}\,\var{i+1} + d\,\var{i}^2
\end{equation}
over the Laplacian ball $K \coloneqq \{X\in\R^n\,\vert\,g(X) \geq 0\}$ 
with $g\coloneqq n-\sum_{i=1}^n (\var{i} - \var{i+1})^2$. 

\subsection{2-D Torus Grid Quartic with Product Cyclic Symmetry}

The product of cyclic groups $\mathfrak{C}_{p}\times \mathfrak{C}_q$ 
of order $n\coloneqq p q$ is generated by commuting cycles $\rho_1$ of length $p$ 
and $\rho_2$ of length $q$, 
whose action on the indeterminates 
$\var{i,j},\,1\leq i\leq p,\,1\leq j\leq q,$ is defined by 
\begin{equation}
    \rho_1(\var{i,j}) = \var{i+1,j} 
    \quad (\mbox{mod } p)
    \quad \mbox{and} \quad
    \rho_2(\var{i,j}) = \var{i,j+1} 
    \quad (\mbox{mod } q). 
\end{equation}
For parameters $a,b,c,d$, consider the quartic 
\begin{equation}
    f
\coloneqq  \sum_{i=1}^p \sum_{j=1}^q
    a\,\var{i,j}^4 + 
    b\,(\var{i,j}^2\,\var{i+1,j}^2 +\var{i,j}^2\,\var{i,j+1}^2) + 
    c\,(\var{i,j}\,\var{i+1,j} +\var{i,j}\,\var{i,j+1}) + 
    d\,\var{i,j}^2 
\end{equation}
over $K \coloneqq \{X\in\R^{p \times q}\,\vert\,h(X) = 0\}$ 
with $h\coloneqq\sum_{i=1}^n \sum_{j=1}^n (\var{i,j}^2-1)^2$, 
that is, $K=\{\pm 1\}^{p \times q}$. 

\subsection{Symmetric Quartic}

The symmetric group $\mathfrak{S}_n$ of order $n!$ is generated by the transpositions $(i,i+1)$, $1\leq i\leq n-1$, and acts by permutation of variables $\var{i}$. 
The quartic 
\begin{equation}
    f   
\coloneqq  \frac{1}{n} \sum\limits_{i=1}^n \var{i}^4 
+   \frac{1}{\binom{n}{4}} \sum\limits_{i<j<k<l} \var{i}\,\var{j}\,\var{k}\,\var{l} 
+   \frac{1}{\binom{n}{3}} \sum\limits_{i<j<k} \var{i}\,\var{j}\,\var{k} 
+   \frac{1}{n} \sum_{i=1}^n \var{i}  
\end{equation}
is $\mathfrak{S}_n$-invariant; see \cite{JohannesThesis}. 

\subsection{Discussion of the Numerical Experiments}

The scatter plots in \Cref{fig_1D_ring,fig_2D_torus,fig_daniel_vs_pop}, 
allow a quantitative comparison of our approach with other existing methods, 
including the dense Lasserre hierarchy, 
denoted by ``\textcolor{black}{$\times$}'', 
from \cite{lasserre01}, 
the hierarchy exploiting sign symmetry, denoted by 
``\textcolor{violet}{\scalebox{1}{$+$}}'', 
the corresponding initial term sparsity hierarchy, denoted by  
``\textcolor{black}{$\blacksquare$}'', 
from \cite{magron21}, as well as 
the symmetry reduction, denoted by 
``\textcolor{teal}{\scalebox{1.5}{$\ast$}}'', 
from \cite{riener2013exploiting}. 

For our method, 
we can choose to include (``\textcolor{red}{\scalebox{1.5}{$\bullet$}}'')  
or discard (``\textcolor{blue}{\scalebox{1.1}{$\blacklozenge$}}'') 
the diagonal squares
\begin{equation}
    \bigcup\limits_{i=1}^{\irreps}
    \bigcup\limits_{j=1}^{\gmultiplicity{1}{r}}
    \supp\left(\reynolds((w^{(i)}_{1,j})^2)\right) 
\end{equation}
when we define $\mathcal{B}_r$ in \Cref{defi_tsp_matrix}. 
Furthermore, when we approximate the block closure/chordal extension, 
we obtain the corresponding SDPs, denoted by 
``\textcolor{black}{$\square$}'', 
``\textcolor{red}{\scalebox{1.5}{$\circ$}}'' and 
``\textcolor{blue}{\scalebox{1.1}{$\lozenge$}}'', respectively. 

In case of the $1$-D ring Ising quartic in \Cref{fig_1D_ring}, 
the best lower bounds are obtained with relaxation order $r=2$ for every method. 
It suffices to consider sparsity order $s=1$ for the methods which exploit sparsity 
(``\textcolor{black}{$\blacksquare$}'', ``\textcolor{red}{\scalebox{1.5}{$\bullet$}}'', ``\textcolor{blue}{\scalebox{1.1}{$\blacklozenge$}}''). 
In terms of maximal SDP block size and number of SDP constraints, 
the least efficient method is the dense one ``\textcolor{black}{$\times$}''. 
The most efficient one on the other hand is our combination ``\textcolor{blue}{\scalebox{1.1}{$\blacklozenge$}}'' of $\mathfrak{D}_{2n}$-symmetry exploitation with term sparsity 
(maximal chordal extension, no diagonal squares). 
Remarkably, the maximal block size seems to stabilize at size $5$. 
Meanwhile, the number of SDP constraints still grows exponentially and is only slightly better than pure symmetry exploitation ``\textcolor{teal}{\scalebox{1.5}{$\ast$}}'' and including diagonal squares ``\textcolor{red}{\scalebox{1.5}{$\bullet$}}''. 

The computation time is to be taken with a grain of salt. 
We indicate the total time, 
which includes the times to compute a symmetry-adapted basis, to compute the block structure, to assemble the SDP and to solve the SDP. 
This strongly depends on the implementation and the machine used for computation, 
but we observe that ``\textcolor{blue}{\scalebox{1.1}{$\blacklozenge$}}'' is also here the winner for increasing the number of variables $n$ and the relaxation order $r$. 
For the dense method ``\textcolor{black}{$\times$}'', 
solving the SDP is the bottleneck, 
while for the methods with symmetry 
(``\textcolor{teal}{\scalebox{1.5}{$\ast$}}'', 
``\textcolor{red}{\scalebox{1.5}{$\circ$}}'', 
``\textcolor{blue}{\scalebox{1.1}{$\lozenge$}}''), 
computing the basis and assembling the SDP takes the most amount of time. 

In case of the $2$-D torus grid quartic in \Cref{fig_2D_torus}, 
the best lower bounds are again obtained with relaxation order $r=2$ for every method. 
However, it is necessary to increase the sparsity order $s\geq 2$ for the methods which exploit sparsity 
(``\textcolor{black}{$\blacksquare$}'', ``\textcolor{red}{\scalebox{1.5}{$\bullet$}}'', ``\textcolor{blue}{\scalebox{1.1}{$\blacklozenge$}}''). 
Again, we remark that the maximal block size stabilizes for the method 
``\textcolor{blue}{\scalebox{1.1}{$\blacklozenge$}}'', 
which is also here the most efficient. 

The final \Cref{fig_daniel_vs_pop} is meant to give a direct comparison with the symmetric chordal extension method ``\textcolor{coral}{\scalebox{1.5}{$\ast$}}'' from \cite{JohannesThesis}. 
While the order of the group $\mathfrak{D}_{2n}$, respectively $\mathfrak{C}_p \times \mathfrak{C}_q$, 
is linear in $n$, respectively $pq$, 
the order of the symmetric group $\mathbf{S}_n$ grows factorially with $n$. 
This makes computing a symmetry-adapted basis and applying the Reynolds operator challenging for large $n$
(for $n=10$, the order is $\vert\mathfrak{S}_{10}\vert=3,628,800$ and there are $h=42$ irreducible representations of dimension up to $768$). 
Tailored algorithms are available \cite{JamesKerber1984}, but we did not apply them here. 
For the case $n=6$, we find similar results as \cite{JohannesThesis}. 
The difference between 
(``\textcolor{coral}{\scalebox{1.5}{$\ast$}}'', $r=2,\,s=2$) 
and
(``\textcolor{red}{\scalebox{1.5}{$\circ$}}'', $r=2,\,s=1$)
is one block of size $1$ in favor of the latter, see \Cref{table_daniel_vs_pop_n6_r2}. 
However, the symmetric chordal extension ``\textcolor{red}{\scalebox{1.5}{$\circ$}}'' appears to give a usable lower bound for $s=1$, 
while the SDPs originating from approximating the maximal chordal extension ``\textcolor{red}{\scalebox{1.5}{$\circ$}}'' and ``\textcolor{blue}{\scalebox{1.1}{$\lozenge$}}'' 
have status ``infeasible'' or ``slow progress'' at the first level of sparsity. 

\clearpage

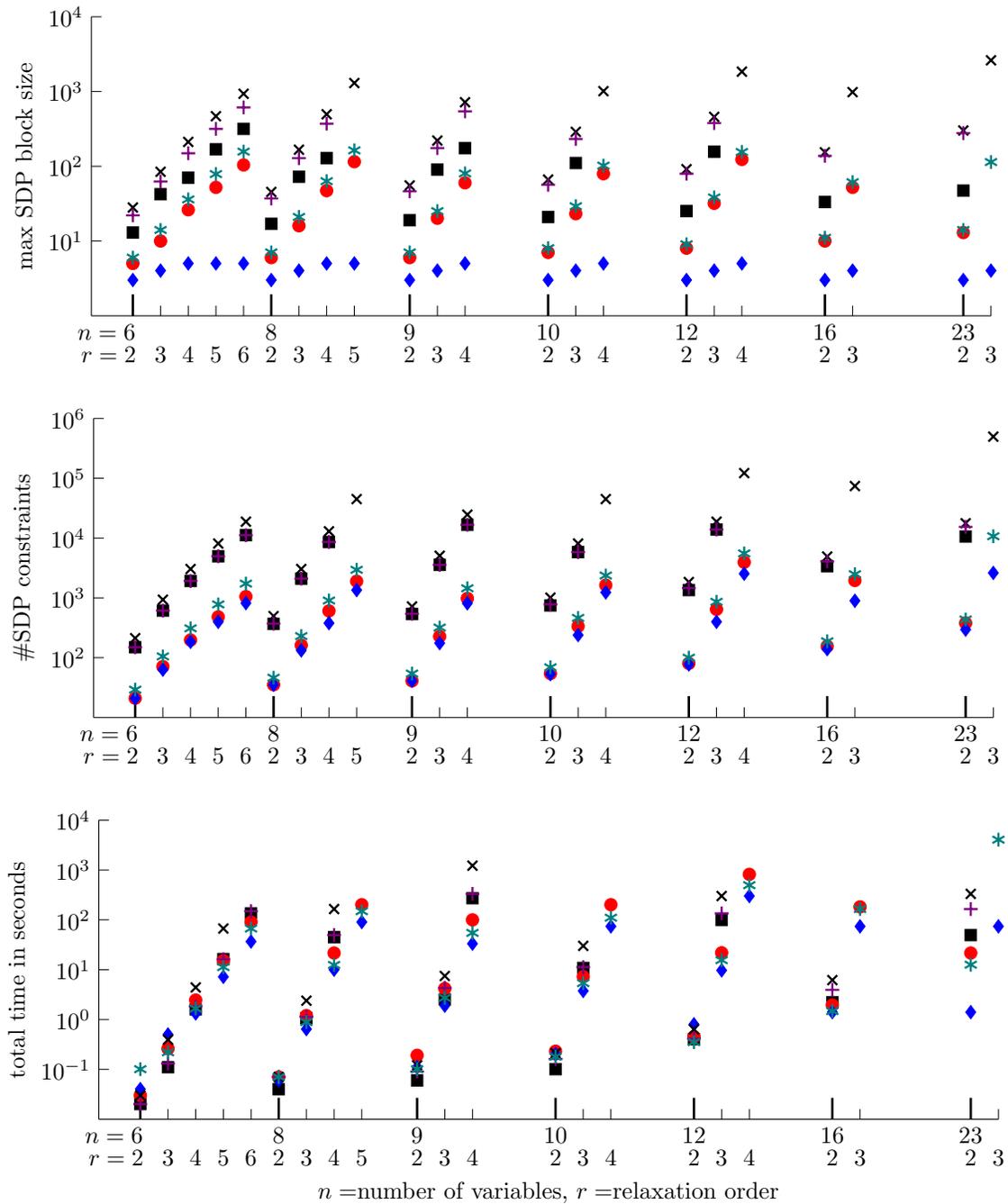
\begin{figure}

\begin{center}
    
\begin{tikzpicture}
\begin{axis}[
			width=0.9\textwidth, height=6cm,
            ylabel={max SDP block size},
			xmin=0.7, xmax=7.2, ymin=0.0, ymax=9.20,
			xtick={1,1.2,1.4,1.6,1.8,2,2.2,2.4,2.6,3,3.2,3.4,4,4.2,4.4,5,5.2,5.4,6,6.2,7,7.2},
            xticklabels={
                {\shortstack{$n=6\phantom{=n..}$\\$r=2\phantom{=r..}$}},
                {\shortstack{$\phantom{1}$\\$3$}},
                {\shortstack{$\phantom{1}$\\$4$}},
                {\shortstack{$\phantom{1}$\\$5$}},
                {\shortstack{$\phantom{1}$\\$6$}},
                {\shortstack{$ 8$\\$2$}},
                {\shortstack{$\phantom{1}$\\$3$}},
                {\shortstack{$\phantom{1}$\\$4$}},
                {\shortstack{$\phantom{1}$\\$5$}},
                {\shortstack{$ 9$\\$2$}},
                {\shortstack{$\phantom{1}$\\$3$}},
                {\shortstack{$\phantom{1}$\\$4$}},
                {\shortstack{$10$\\$2$}},
                {\shortstack{$\phantom{1}$\\$3$}},
                {\shortstack{$\phantom{1}$\\$4$}},
                {\shortstack{$12$\\$2$}},
                {\shortstack{$\phantom{1}$\\$3$}},
                {\shortstack{$\phantom{1}$\\$4$}},
                {\shortstack{$16$\\$2$}},
                {\shortstack{$\phantom{1}$\\$3$}},
                {\shortstack{$23$\\$2$}},
                {\shortstack{$\phantom{1}$\\$3$}}
            },
            extra x ticks={1.00,2.00,3.00,4.00,5.00,6.00,7.00},
            extra x tick labels={$\phantom{2}$,$\phantom{2}$,$\phantom{2}$,$\phantom{2}$,$\phantom{2}$,$\phantom{2}$,$\phantom{2}$},
            extra x tick style={
                major tick length=9pt,
                tick style={line width=1pt},
                xticklabel style={yshift=-2ex, font=\small, color=black}
            },
			ytick={2.30,4.60,6.90,9.20,11.5,13.8,16.1},
            yticklabels={$10^1$, $10^2$, $10^3$, $10^4$, $10^5$, $10^6$, $10^7$},
			axis x line*=bottom,
			axis y line*=left,
			enlargelimits=false,
			legend cell align=left,
			legend pos=north west,
			tick style={black},
			]
			
			\addplot[
			only marks, mark=square*, mark size=2.4pt,
			draw=black, fill=black
			] coordinates {
				(1.0, 2.56) (1.2, 3.74) (1.4, 4.25) (1.6, 5.12) (1.8, 5.75) (2.0, 2.83) (2.2, 4.28) (2.4, 4.85) (3.0, 2.94) (3.2, 4.50) (3.4, 5.16) (4.0, 3.04) (4.2, 4.70) (5.0, 3.22) (5.2, 5.05) (6.0, 3.50) (7.0, 3.85)
			};
			
			\addplot[
			only marks, mark=*, mark size=2.6pt,
			draw=red, fill=red
			] coordinates {
				(1.0, 1.61) (1.2, 2.30) (1.4, 3.26) (1.6, 3.95) (1.8, 4.64) (2.0, 1.79) (2.2, 2.77) (2.4, 3.85) (2.6, 4.74) (3.0, 1.79) (3.2, 3.00) (3.4, 4.09) (4.0, 1.95) (4.2, 3.14) (4.4, 4.37) (5.0, 2.08) (5.2, 3.46) (5.4, 4.81) (6.0, 2.30) (6.2, 3.95) (7.0, 2.56)
			};
			
			\addplot[
			only marks, mark=diamond*, mark size=2.8pt,
			draw=blue, fill=blue
			] coordinates {
				(1.0, 1.10) (1.2, 1.39) (1.4, 1.61) (1.6, 1.61) (1.8, 1.61) (2.0, 1.10) (2.2, 1.39) (2.4, 1.61) (2.6, 1.61) (3.0, 1.10) (3.2, 1.39) (3.4, 1.61) (4.0, 1.10) (4.2, 1.39) (4.4, 1.61) (5.0, 1.10) (5.2, 1.39) (5.4, 1.61) (6.0, 1.10) (6.2, 1.39) (7.0, 1.10) (7.2, 1.39)
			};
			
			\addplot[
			only marks, mark=x, mark size=2.9pt,
			draw=black, line width=0.9pt
			] coordinates {
				(1.0, 3.33) (1.2, 4.43) (1.4, 5.35) (1.6, 6.14) (1.8, 6.83) (2.0, 3.81) (2.2, 5.11) (2.4, 6.20) (2.6, 7.16) (3.0, 4.01) (3.2, 5.39) (3.4, 6.57) (4.0, 4.19) (4.2, 5.66) (4.4, 6.91) (5.0, 4.51) (5.2, 6.12) (5.4, 7.51) (6.0, 5.03) (6.2, 6.88) (7.0, 5.70) (7.2, 7.86)
			};
			
			\addplot[
			only marks, mark=+, mark size=2.9pt,
			draw=violet, line width=0.9pt
			] coordinates {
				(1.0, 3.09) (1.2, 4.13) (1.4, 5.00) (1.6, 5.75) (1.8, 6.41) (2.0, 3.61) (2.2, 4.85) (2.4, 5.91) (3.0, 3.83) (3.2, 5.16) (3.4, 6.29) (4.0, 4.03) (4.2, 5.44) (5.0, 4.37) (5.2, 5.93) (6.0, 4.92) (7.0, 5.62)
			};
			
			\addplot[
			only marks, mark=asterisk, mark size=2.9pt,
			teal, line width=0.9pt
			] coordinates {
				(1.0, 1.79) (1.2, 2.64) (1.4, 3.58) (1.6, 4.36) (1.8, 5.05) (2.0, 1.95) (2.2, 3.04) (2.4, 4.14) (2.6, 5.09) (3.0, 1.95) (3.2, 3.22) (3.4, 4.38) (4.0, 2.08) (4.2, 3.37) (4.4, 4.62) (5.0, 2.20) (5.2, 3.64) (5.4, 5.04) (6.0, 2.40) (6.2, 4.11) (7.0, 2.64) (7.2, 4.73)
			};
		\end{axis}
\end{tikzpicture}

~

\begin{tikzpicture}
\begin{axis}[
			width=0.9\textwidth, height=6cm,
            ylabel={$\#$SDP constraints},
			xmin=0.7, xmax=7.2, ymin=2.30, ymax=13.8,
			xtick={1,1.2,1.4,1.6,1.8,2,2.2,2.4,2.6,3,3.2,3.4,4,4.2,4.4,5,5.2,5.4,6,6.2,7,7.2},
            xticklabels={
                {\shortstack{$n=6\phantom{=n..}$\\$r=2\phantom{=r..}$}},
                {\shortstack{$\phantom{1}$\\$3$}},
                {\shortstack{$\phantom{1}$\\$4$}},
                {\shortstack{$\phantom{1}$\\$5$}},
                {\shortstack{$\phantom{1}$\\$6$}},
                {\shortstack{$ 8$\\$2$}},
                {\shortstack{$\phantom{1}$\\$3$}},
                {\shortstack{$\phantom{1}$\\$4$}},
                {\shortstack{$\phantom{1}$\\$5$}},
                {\shortstack{$ 9$\\$2$}},
                {\shortstack{$\phantom{1}$\\$3$}},
                {\shortstack{$\phantom{1}$\\$4$}},
                {\shortstack{$10$\\$2$}},
                {\shortstack{$\phantom{1}$\\$3$}},
                {\shortstack{$\phantom{1}$\\$4$}},
                {\shortstack{$12$\\$2$}},
                {\shortstack{$\phantom{1}$\\$3$}},
                {\shortstack{$\phantom{1}$\\$4$}},
                {\shortstack{$16$\\$2$}},
                {\shortstack{$\phantom{1}$\\$3$}},
                {\shortstack{$23$\\$2$}},
                {\shortstack{$\phantom{1}$\\$3$}}
            },
            extra x ticks={1.00,2.00,3.00,4.00,5.00,6.00,7.00},
            extra x tick labels={$\phantom{2}$,$\phantom{2}$,$\phantom{2}$,$\phantom{2}$,$\phantom{2}$,$\phantom{2}$,$\phantom{2}$},
            extra x tick style={
                major tick length=9pt,
                tick style={line width=1pt},
                xticklabel style={yshift=-2ex, font=\small, color=black}
            },
			ytick={4.60,6.90,9.20,11.5,13.8,16.1},
            yticklabels={$10^2$, $10^3$, $10^4$, $10^5$, $10^6$, $10^7$},
			axis x line*=bottom,
			axis y line*=left,
			enlargelimits=false,
			legend cell align=left,
			legend pos=north west,
			tick style={black},
			]
			
			\addplot[
			only marks, mark=square*, mark size=2.4pt,
			draw=black, fill=black
			] coordinates {
				(1.0, 5.00) (1.2, 6.41) (1.4, 7.55) (1.6, 8.50) (1.8, 9.31) (2.0, 5.90) (2.2, 7.64) (2.4, 9.05) (3.0, 6.28) (3.2, 8.17) (3.4, 9.71) (4.0, 6.61) (4.2, 8.66) (5.0, 7.20) (5.2, 9.53) (6.0, 8.12) (7.0, 9.26)
			};
			
			\addplot[
			only marks, mark=*, mark size=2.6pt,
			draw=red, fill=red
			] coordinates {
				(1.0, 3.04) (1.2, 4.26) (1.4, 5.28) (1.6, 6.17) (1.8, 6.95) (2.0, 3.56) (2.2, 5.08) (2.4, 6.40) (2.6, 7.54) (3.0, 3.71) (3.2, 5.42) (3.4, 6.88) (4.0, 3.99) (4.2, 5.81) (4.4, 7.39) (5.0, 4.38) (5.2, 6.46) (5.4, 8.27) (6.0, 5.04) (6.2, 7.57) (7.0, 5.93)
			};
			
			\addplot[
			only marks, mark=diamond*, mark size=2.8pt,
			draw=blue, fill=blue
			] coordinates {
				(1.0, 3.04) (1.2, 4.14) (1.4, 5.21) (1.6, 5.98) (1.8, 6.70) (2.0, 3.56) (2.2, 4.88) (2.4, 5.93) (2.6, 7.20) (3.0, 3.71) (3.2, 5.16) (3.4, 6.69) (4.0, 3.97) (4.2, 5.47) (4.4, 7.11) (5.0, 4.34) (5.2, 5.98) (5.4, 7.83) (6.0, 4.93) (6.2, 6.79) (7.0, 5.68) (7.2, 7.86)
			};
			
			\addplot[
			only marks, mark=x, mark size=2.9pt,
			draw=black, line width=0.9pt
			] coordinates {
				(1.0, 5.35) (1.2, 6.83) (1.4, 8.01) (1.6, 8.99) (1.8, 9.83) (2.0, 6.20) (2.2, 8.01) (2.4, 9.46) (2.6, 10.7) (3.0, 6.57) (3.2, 8.52) (3.4, 10.1) (4.0, 6.91) (4.2, 8.99) (4.4, 10.7) (5.0, 7.51) (5.2, 9.83) (5.4, 11.7) (6.0, 8.49) (6.2, 11.2) (7.0, 9.77) (7.2, 13.1)
			};
			
			\addplot[
			only marks, mark=+, mark size=2.9pt,
			draw=violet, line width=0.9pt
			] coordinates {
				(1.0, 5.00) (1.2, 6.41) (1.4, 7.55) (1.6, 8.50) (1.8, 9.31) (2.0, 5.91) (2.2, 7.64) (2.4, 9.05) (3.0, 6.29) (3.2, 8.17) (3.4, 9.71) (4.0, 6.65) (4.2, 8.66) (5.0, 7.28) (5.2, 9.53) (6.0, 8.30) (7.0, 9.63)
			};
			
			\addplot[
			only marks, mark=asterisk, mark size=2.9pt,
			teal, line width=0.9pt
			] coordinates {
				(1.0, 3.37) (1.2, 4.65) (1.4, 5.73) (1.6, 6.65) (1.8, 7.45) (2.0, 3.83) (2.2, 5.43) (2.4, 6.80) (2.6, 7.98) (3.0, 3.99) (3.2, 5.76) (3.4, 7.27) (4.0, 4.23) (4.2, 6.13) (4.4, 7.76) (5.0, 4.60) (5.2, 6.75) (5.4, 8.61) (6.0, 5.23) (6.2, 7.81) (7.0, 6.07) (7.2, 9.27)
			};
		\end{axis}
\end{tikzpicture}

~

\begin{tikzpicture}
\begin{axis}[
			width=0.9\textwidth, height=6cm,
            xlabel={$n=$number of variables, $r=$relaxation order},
            ylabel={total time in seconds},
			xmin=0.7, xmax=7.2, ymin=2.30, ymax=16.10,
			xtick={1,1.2,1.4,1.6,1.8,2,2.2,2.4,2.6,3,3.2,3.4,4,4.2,4.4,5,5.2,5.4,6,6.2,7,7.2},
            xticklabels={
                {\shortstack{$n=6\phantom{=n..}$\\$r=2\phantom{=r..}$}},
                {\shortstack{$\phantom{1}$\\$3$}},
                {\shortstack{$\phantom{1}$\\$4$}},
                {\shortstack{$\phantom{1}$\\$5$}},
                {\shortstack{$\phantom{1}$\\$6$}},
                {\shortstack{$ 8$\\$2$}},
                {\shortstack{$\phantom{1}$\\$3$}},
                {\shortstack{$\phantom{1}$\\$4$}},
                {\shortstack{$\phantom{1}$\\$5$}},
                {\shortstack{$ 9$\\$2$}},
                {\shortstack{$\phantom{1}$\\$3$}},
                {\shortstack{$\phantom{1}$\\$4$}},
                {\shortstack{$10$\\$2$}},
                {\shortstack{$\phantom{1}$\\$3$}},
                {\shortstack{$\phantom{1}$\\$4$}},
                {\shortstack{$12$\\$2$}},
                {\shortstack{$\phantom{1}$\\$3$}},
                {\shortstack{$\phantom{1}$\\$4$}},
                {\shortstack{$16$\\$2$}},
                {\shortstack{$\phantom{1}$\\$3$}},
                {\shortstack{$23$\\$2$}},
                {\shortstack{$\phantom{1}$\\$3$}}
            },
            extra x ticks={1.00,2.00,3.00,4.00,5.00,6.00,7.00},
            extra x tick labels={$\phantom{2}$,$\phantom{2}$,$\phantom{2}$,$\phantom{2}$,$\phantom{2}$,$\phantom{2}$,$\phantom{2}$},
            extra x tick style={
                major tick length=9pt,
                tick style={line width=1pt},
                xticklabel style={yshift=-2ex, font=\small, color=black}
            },
			ytick={4.60,6.90,9.20,11.5,13.8,16.1},
            yticklabels={$10^{-1}$, $10^0$, $10^1$, $10^2$, $10^3$, $10^4$},
			axis x line*=bottom,
			axis y line*=left,
			enlargelimits=false,
			legend cell align=left,
			legend pos=north west,
			tick style={black},
			]
			
			\addplot[
			only marks, mark=square*, mark size=2.4pt,
			draw=black, fill=black
			] coordinates {
				(1.0, 3.00) (1.2, 4.70) (1.4, 7.37) (1.6, 9.69) (1.8, 11.8) (2.0, 3.69) (2.2, 6.96) (2.4, 10.7) (3.0, 4.09) (3.2, 7.83) (3.4, 12.5) (4.0, 4.61) (4.2, 9.28) (5.0, 5.99) (5.2, 11.5) (6.0, 7.70) (7.0, 10.8)
			};
			
			\addplot[
			only marks, mark=*, mark size=2.6pt,
			draw=red, fill=red
			] coordinates {
				(1.0, 3.40) (1.2, 5.56) (1.4, 7.80) (1.6, 9.62) (1.8, 11.4) (2.0, 4.25) (2.2, 7.08) (2.4, 9.97) (2.6, 12.2) (3.0, 5.25) (3.2, 8.33) (3.4, 11.5) (4.0, 5.44) (4.2, 8.88) (4.4, 12.2) (5.0, 6.11) (5.2, 9.98) (5.4, 13.6) (6.0, 7.57) (6.2, 12.1) (7.0, 9.97)
			};
			
			\addplot[
			only marks, mark=diamond*, mark size=2.8pt,
			draw=blue, fill=blue
			] coordinates {
				(1.0, 3.69) (1.2, 6.21) (1.4, 7.18) (1.6, 8.87) (1.8, 10.5) (2.0, 4.09) (2.2, 6.46) (2.4, 9.20) (2.6, 11.4) (3.0, 4.70) (3.2, 7.54) (3.4, 10.4) (4.0, 5.30) (4.2, 8.22) (4.4, 11.2) (5.0, 6.68) (5.2, 9.17) (5.4, 12.6) (6.0, 7.24) (6.2, 11.2) (7.0, 7.24) (7.2, 11.2)
			};
			
			\addplot[
			only marks, mark=x, mark size=2.9pt,
			draw=black, line width=0.9pt
			] coordinates {
				(1.0, 3.40) (1.2, 5.97) (1.4, 8.38) (1.6, 11.1) (2.0, 4.25) (2.2, 7.77) (2.4, 12.0) (3.0, 4.79) (3.2, 8.91) (3.4, 14.0) (4.0, 5.35) (4.2, 10.3) (5.0, 6.45) (5.2, 12.6) (6.0, 8.72) (7.0, 12.7)
			};
			
			\addplot[
			only marks, mark=+, mark size=2.9pt,
			draw=violet, line width=0.9pt
			] coordinates {
				(1.0, 3.00) (1.2, 4.87) (1.4, 7.51) (1.6, 9.68) (1.8, 11.9) (2.0, 4.25) (2.2, 7.02) (2.4, 10.8) (3.0, 4.50) (3.2, 8.36) (3.4, 12.7) (4.0, 5.08) (4.2, 9.32) (5.0, 6.02) (5.2, 11.8) (6.0, 8.27) (7.0, 12.0)
			};
			
			\addplot[
			only marks, mark=asterisk, mark size=2.9pt,
			teal, line width=0.9pt
			] coordinates {
				(1.0, 4.61) (1.2, 5.39) (1.4, 7.41) (1.6, 9.32) (1.8, 11.1) (2.0, 4.25) (2.2, 6.78) (2.4, 9.42) (2.6, 11.9) (3.0, 4.61) (3.2, 7.90) (3.4, 10.9) (4.0, 5.19) (4.2, 8.58) (4.4, 11.6) (5.0, 5.86) (5.2, 9.64) (5.4, 13.1) (6.0, 7.29) (6.2, 12.0) (7.0, 9.43) (7.2, 15.2)
			};
		\end{axis}
\end{tikzpicture}
    
\end{center}
\caption{Maximal SDP block sizes, number of SDP constraints and total computation time versus $(n,\,r)$ for the $1$-D ring Ising quartic with dihedral symmetry.}
\label{fig_1D_ring}

\textcolor{black}{$\times$ 
dense \cite{lasserre01}}\\
\textcolor{violet}{\scalebox{1}{$+$} 
sign symmetry}\\
\textcolor{teal}{\scalebox{1.5}{$\ast$} 
$\mathfrak{D}_{2n}$-symmetry \cite{riener2013exploiting}}\\
\textcolor{black}{$\blacksquare$ 
term sparsity (maximal chordal extension, diagonal squares) \cite{magron21}}\\
\textcolor{red}{\scalebox{1.5}{$\bullet$}
$\mathfrak{D}_{2n}$-symmetry + term sparsity (maximal chordal extension, diagonal squares)}\\
\textcolor{blue}{\scalebox{1.1}{$\blacklozenge$}
$\mathfrak{D}_{2n}$-symmetry + term sparsity (maximal chordal extension, no diagonal squares)}

\end{figure}

\begin{figure}
\begin{center}
    
\begin{tikzpicture}
\begin{axis}[
			width=0.9\textwidth, height=6cm,
			xmin=0.95, xmax=2.40,
			xtick={
            1.00,1.05,
            1.15,1.20,
            1.30,1.35,
            1.45,1.50,1.55,
            1.60,1.65,1.70,
            1.75,1.80,1.85,
            2.00,2.05,
            2.15,2.20,
            2.30,2.35,2.40
            },
            xticklabels={
            {\shortstack{$p,q=2,3\phantom{=p,q.}$\\$r=2\phantom{=r.}$\\$s=1\phantom{=s.}$}},
            {\shortstack{$\phantom{2,3}$\\$\phantom{2}$\\$2$}},
            {\shortstack{$\phantom{2,3}$\\$         3 $\\$1$}},
            {\shortstack{$\phantom{2,3}$\\$\phantom{3}$\\$2$}},
            {\shortstack{$\phantom{2,3}$\\$         4 $\\$1$}},
            {\shortstack{$\phantom{2,3}$\\$\phantom{4}$\\$2$}},
            {\shortstack{$\phantom{2,3}$\\$         5 $\\$1$}},
            {\shortstack{$\phantom{2,3}$\\$\phantom{5}$\\$2$}},
            {\shortstack{$\phantom{2,3}$\\$\phantom{5}$\\$3$}},
            {\shortstack{$\phantom{2,3}$\\$         6 $\\$1$}},
            {\shortstack{$\phantom{2,3}$\\$\phantom{6}$\\$2$}},
            {\shortstack{$\phantom{2,3}$\\$\phantom{6}$\\$3$}},
            {\shortstack{$\phantom{2,3}$\\$         7 $\\$1$}},
            {\shortstack{$\phantom{2,3}$\\$\phantom{7}$\\$2$}},
            {\shortstack{$\phantom{2,3}$\\$\phantom{7}$\\$3$}},
            {\shortstack{$         3,4 $\\$         2 $\\$1$}},
            {\shortstack{$\phantom{3,4}$\\$\phantom{2}$\\$2$}},
            {\shortstack{$\phantom{3,4}$\\$         3 $\\$1$}},
            {\shortstack{$\phantom{2,3}$\\$\phantom{3}$\\$2$}},
            {\shortstack{$\phantom{3,4}$\\$         4 $\\$1$}},
            {\shortstack{$\phantom{2,3}$\\$\phantom{4}$\\$2$}},
            {\shortstack{$\phantom{2,3}$\\$\phantom{4}$\\$3$}}
            },
            extra x ticks={1.00,2.00},
            extra x tick labels={$\phantom{2}$,$\phantom{2}$},
            extra x tick style={
                major tick length=9pt,
                tick style={line width=1pt},
                xticklabel style={yshift=-2ex, font=\small, color=black}
            },
            ylabel={max SDP block size},
            ymin=0.0, ymax=9.2, 
			ytick={2.30,4.60,6.90,9.20,11.5,13.8,16.1},
            yticklabels={$10^1$, $10^2$, $10^3$, $10^4$, $10^5$, $10^6$, $10^7$},
			axis x line*=bottom,
			axis y line*=left,
			enlargelimits=false,
			legend cell align=left,
			legend pos=north west,
			tick style={black},
			]
			
			\addplot[
			only marks, mark=square*, mark size=2.4pt,
			draw=black, fill=black
			] coordinates {
				(1.00, 2.77) (1.05, 3.09) (1.15, 3.74) (1.20, 4.13) (1.30, 4.51) (1.35, 5.00) (1.45, 5.12) (1.50, 5.75) (1.55, 5.75) (1.60, 5.82) (1.65, 6.41) (1.70, 6.41) (2.00, 3.61) (2.05, 4.37) (2.15, 5.05) (2.20, 5.93)
			};
			
			\addplot[
			only marks, mark=*, mark size=2.6pt,
			draw=red, fill=red
			] coordinates {
				(1.00, 1.61) (1.05, 2.08) (1.15, 3.00) (1.20, 3.00) (1.30, 3.95) (1.45, 4.64) (1.50, 4.64) (1.60, 5.34) (1.65, 5.34) (1.75, 5.91) (1.80, 5.91) (2.00, 2.08) (2.05, 2.64) (2.15, 4.16) (2.20, 4.16) (2.30, 5.51) (2.35, 5.51)
			};
			
			\addplot[
			only marks, mark=diamond*, mark size=2.8pt,
			draw=blue, fill=blue
			] coordinates {
				(1.00, 1.39) (1.05, 2.08) (1.15, 2.30) (1.20, 3.00) (1.30, 2.30) (1.35, 3.95) (1.45, 2.30) (1.50, 4.64) (1.55, 4.64) (1.60, 2.30) (1.65, 5.34) (1.70, 5.34) (1.75, 2.30) (1.80, 5.91) (1.85, 5.91) (2.00, 1.39) (2.05, 2.64) (2.15, 2.48) (2.20, 4.16) (2.30, 2.48) (2.35, 5.51) (2.40, 5.51)
			};
			
			\addplot[
			only marks, mark=x, mark size=2.9pt,
			draw=black, line width=0.9pt
			] coordinates {
				(1.00, 3.33) (1.15, 4.43) (1.30, 5.35) (1.45, 6.14) (1.60, 6.83) (1.75, 7.45) (2.00, 4.51) (2.15, 6.12) (2.30, 7.51)
			};
			
			\addplot[
			only marks, mark=+, mark size=2.9pt,
			draw=violet, line width=0.9pt
			] coordinates {
				(1.00, 3.09) (1.15, 4.13) (1.30, 5.00) (1.45, 5.75) (1.60, 6.41) (2.00, 4.37) (2.15, 5.93)
			};
			
			\addplot[
			only marks, mark=asterisk, mark size=2.9pt,
			teal, line width=0.9pt
			] coordinates {
				(1.00, 2.30) (1.15, 3.33) (1.30, 4.28) (1.45, 5.05) (1.60, 5.74) (1.75, 6.36) (2.00, 2.77) (2.15, 4.33) (2.30, 5.73)
			};
		\end{axis}
\end{tikzpicture}

~

\begin{tikzpicture}
\begin{axis}[
			width=0.9\textwidth, height=6cm,
			xmin=0.95, xmax=2.40,
			xtick={
            1.00,1.05,
            1.15,1.20,
            1.30,1.35,
            1.45,1.50,1.55,
            1.60,1.65,1.70,
            1.75,1.80,1.85,
            2.00,2.05,
            2.15,2.20,
            2.30,2.35,2.40
            },
            xticklabels={
            {\shortstack{$p,q=2,3\phantom{=p,q.}$\\$r=2\phantom{=r.}$\\$s=1\phantom{=s.}$}},
            {\shortstack{$\phantom{2,3}$\\$\phantom{2}$\\$2$}},
            {\shortstack{$\phantom{2,3}$\\$         3 $\\$1$}},
            {\shortstack{$\phantom{2,3}$\\$\phantom{3}$\\$2$}},
            {\shortstack{$\phantom{2,3}$\\$         4 $\\$1$}},
            {\shortstack{$\phantom{2,3}$\\$\phantom{4}$\\$2$}},
            {\shortstack{$\phantom{2,3}$\\$         5 $\\$1$}},
            {\shortstack{$\phantom{2,3}$\\$\phantom{5}$\\$2$}},
            {\shortstack{$\phantom{2,3}$\\$\phantom{5}$\\$3$}},
            {\shortstack{$\phantom{2,3}$\\$         6 $\\$1$}},
            {\shortstack{$\phantom{2,3}$\\$\phantom{6}$\\$2$}},
            {\shortstack{$\phantom{2,3}$\\$\phantom{6}$\\$3$}},
            {\shortstack{$\phantom{2,3}$\\$         7 $\\$1$}},
            {\shortstack{$\phantom{2,3}$\\$\phantom{7}$\\$2$}},
            {\shortstack{$\phantom{2,3}$\\$\phantom{7}$\\$3$}},
            {\shortstack{$         3,4 $\\$         2 $\\$1$}},
            {\shortstack{$\phantom{3,4}$\\$\phantom{2}$\\$2$}},
            {\shortstack{$\phantom{3,4}$\\$         3 $\\$1$}},
            {\shortstack{$\phantom{2,3}$\\$\phantom{3}$\\$2$}},
            {\shortstack{$\phantom{3,4}$\\$         4 $\\$1$}},
            {\shortstack{$\phantom{2,3}$\\$\phantom{4}$\\$2$}},
            {\shortstack{$\phantom{2,3}$\\$\phantom{4}$\\$3$}}
            },
            extra x ticks={1.00,2.00},
            extra x tick labels={$\phantom{2}$,$\phantom{2}$},
            extra x tick style={
                major tick length=9pt,
                tick style={line width=1pt},
                xticklabel style={yshift=-2ex, font=\small, color=black}
            },
            ylabel={$\#$SDP constraints},
            ymin=2.30, ymax=13.8, 
			ytick={4.60,6.90,9.20,11.5,13.8,16.1},
            yticklabels={$10^2$, $10^3$, $10^4$, $10^5$, $10^6$, $10^7$},
			axis x line*=bottom,
			axis y line*=left,
			enlargelimits=false,
			legend cell align=left,
			legend pos=north west,
			tick style={black},
			]
			
			\addplot[
			only marks, mark=square*, mark size=2.4pt,
			draw=black, fill=black
			] coordinates {
				(1.00, 4.82) (1.05, 5.00) (1.15, 6.25) (1.20, 6.41) (1.30, 7.48) (1.35, 7.55) (1.45, 8.30) (1.50, 8.50) (1.55, 8.50) (1.60, 9.28) (1.65, 9.31) (1.70, 9.31) (2.00, 6.54) (2.05, 7.28) (2.15, 8.80) (2.20, 9.53)
			};
			
			\addplot[
			only marks, mark=*, mark size=2.6pt,
			draw=red, fill=red
			] coordinates {
				(1.00, 3.30) (1.05, 3.30) (1.15, 4.67) (1.20, 4.67) (1.30, 5.78) (1.45, 6.72) (1.50, 6.72) (1.60, 7.53) (1.65, 7.53) (1.75, 8.25) (1.80, 8.25) (2.00, 4.82) (2.05, 4.82) (2.15, 7.06) (2.20, 7.06) (2.30, 8.92) (2.35, 8.92)
			};
			
			\addplot[
			only marks, mark=diamond*, mark size=2.8pt,
			draw=blue, fill=blue
			] coordinates {
				(1.00, 3.14) (1.05, 3.30) (1.15, 4.47) (1.20, 4.67) (1.30, 4.96) (1.35, 5.78) (1.45, 5.60) (1.50, 6.72) (1.55, 6.72) (1.60, 6.14) (1.65, 7.53) (1.70, 7.53) (1.75, 6.74) (1.80, 8.25) (1.85, 8.25) (2.00, 4.20) (2.05, 4.82) (2.15, 5.98) (2.20, 7.06) (2.30, 6.91) (2.35, 8.92) (2.40, 8.92)
			};
			
			\addplot[
			only marks, mark=x, mark size=2.9pt,
			draw=black, line width=0.9pt
			] coordinates {
				(1.00, 5.35) (1.15, 6.83) (1.30, 8.01) (1.45, 8.99) (1.60, 9.83) (1.75, 10.6) (2.00, 7.51) (2.15, 9.83) (2.30, 11.7)
			};
			
			\addplot[
			only marks, mark=+, mark size=2.9pt,
			draw=violet, line width=0.9pt
			] coordinates {
				(1.00, 5.00) (1.15, 6.41) (1.30, 7.55) (1.45, 8.50) (1.60, 9.31) (2.00, 7.28) (2.15, 9.53)
			};
			
			\addplot[
			only marks, mark=asterisk, mark size=2.9pt,
			teal, line width=0.9pt
			] coordinates {
				(1.00, 3.64) (1.15, 5.08) (1.30, 6.23) (1.45, 7.21) (1.60, 8.04) (1.75, 8.78) (2.00, 5.05) (2.15, 7.35) (2.30, 9.26)
			};
		\end{axis}
\end{tikzpicture}

~

\begin{tikzpicture}
\begin{axis}[
			width=0.9\textwidth, height=6cm,
			xmin=0.95, xmax=2.40,
            xlabel={$pq=$number of variables, $r=$relaxation order, $s=$sparsity order},
			xtick={
            1.00,1.05,
            1.15,1.20,
            1.30,1.35,
            1.45,1.50,1.55,
            1.60,1.65,1.70,
            1.75,1.80,1.85,
            2.00,2.05,
            2.15,2.20,
            2.30,2.35,2.40
            },
            xticklabels={
            {\shortstack{$p,q=2,3\phantom{=p,q.}$\\$r=2\phantom{=r.}$\\$s=1\phantom{=s.}$}},
            {\shortstack{$\phantom{2,3}$\\$\phantom{2}$\\$2$}},
            {\shortstack{$\phantom{2,3}$\\$         3 $\\$1$}},
            {\shortstack{$\phantom{2,3}$\\$\phantom{3}$\\$2$}},
            {\shortstack{$\phantom{2,3}$\\$         4 $\\$1$}},
            {\shortstack{$\phantom{2,3}$\\$\phantom{4}$\\$2$}},
            {\shortstack{$\phantom{2,3}$\\$         5 $\\$1$}},
            {\shortstack{$\phantom{2,3}$\\$\phantom{5}$\\$2$}},
            {\shortstack{$\phantom{2,3}$\\$\phantom{5}$\\$3$}},
            {\shortstack{$\phantom{2,3}$\\$         6 $\\$1$}},
            {\shortstack{$\phantom{2,3}$\\$\phantom{6}$\\$2$}},
            {\shortstack{$\phantom{2,3}$\\$\phantom{6}$\\$3$}},
            {\shortstack{$\phantom{2,3}$\\$         7 $\\$1$}},
            {\shortstack{$\phantom{2,3}$\\$\phantom{7}$\\$2$}},
            {\shortstack{$\phantom{2,3}$\\$\phantom{7}$\\$3$}},
            {\shortstack{$         3,4 $\\$         2 $\\$1$}},
            {\shortstack{$\phantom{3,4}$\\$\phantom{2}$\\$2$}},
            {\shortstack{$\phantom{3,4}$\\$         3 $\\$1$}},
            {\shortstack{$\phantom{2,3}$\\$\phantom{3}$\\$2$}},
            {\shortstack{$\phantom{3,4}$\\$         4 $\\$1$}},
            {\shortstack{$\phantom{2,3}$\\$\phantom{4}$\\$2$}},
            {\shortstack{$\phantom{2,3}$\\$\phantom{4}$\\$3$}}
            },
            extra x ticks={1.00,2.00},
            extra x tick labels={$\phantom{2}$,$\phantom{2}$},
            extra x tick style={
                major tick length=9pt,
                tick style={line width=1pt},
                xticklabel style={yshift=-2ex, font=\small, color=black}
            },
            ylabel={total time in seconds},
            ymin=2.30, ymax=13.8, 
			ytick={4.60,6.90,9.20,11.5,13.8,16.1},
            yticklabels={$10^{-1}$, $10^0$, $10^1$, $10^2$, $10^3$, $10^4$},
			axis x line*=bottom,
			axis y line*=left,
			enlargelimits=false,
			legend cell align=left,
			legend pos=north west,
			tick style={black},
			]
			
			\addplot[
			only marks, mark=square*, mark size=2.4pt,
			draw=black, fill=black
			] coordinates {
				(1.00, 3.69) (1.05, 5.14) (1.15, 5.67) (1.20, 6.45) (1.30, 7.79) (1.35, 7.91) (1.45, 9.00) (1.50, 10.1) (1.55, 9.70) (1.60, 11.9) (1.65, 12.8) (1.70, 12.0) (2.00, 5.70) (2.05, 7.06) (2.15, 10.9) (2.20, 13.3)
			};
			
			\addplot[
			only marks, mark=*, mark size=2.6pt,
			draw=red, fill=red
			] coordinates {
				(1.00, 4.38) (1.05, 4.25) (1.15, 5.14) (1.20, 5.99) (1.30, 6.84) (1.45, 8.35) (1.50, 8.36) (1.60, 10.1) (1.65, 10.7) (1.75, 12.0) (1.80, 11.8) (2.00, 6.85) (2.05, 5.97) (2.15, 9.20) (2.20, 9.28) (2.30, 12.5) (2.35, 12.8)
			};
			
			\addplot[
			only marks, mark=diamond*, mark size=2.8pt,
			draw=blue, fill=blue
			] coordinates {
				(1.00, 3.00) (1.05, 4.38) (1.15, 4.87) (1.20, 5.08) (1.30, 5.80) (1.35, 6.96) (1.45, 7.03) (1.50, 8.58) (1.55, 8.47) (1.60, 8.37) (1.65, 10.3) (1.70, 10.7) (1.75, 9.78) (1.80, 12.0) (1.85, 11.9) (2.00, 5.39) (2.05, 5.99) (2.15, 8.11) (2.20, 9.15) (2.30, 10.8) (2.35, 13.0) (2.40, 12.9)
			};
			
			\addplot[
			only marks, mark=x, mark size=2.9pt,
			draw=black, line width=0.9pt
			] coordinates {
				(1.00, 5.97) (1.15, 6.62) (1.30, 9.40) (1.45, 12.3) (2.00, 8.25) (2.15, 13.7)
			};
			
			\addplot[
			only marks, mark=+, mark size=2.9pt,
			draw=violet, line width=0.9pt
			] coordinates {
				(1.00, 5.19) (1.15, 6.33) (1.30, 7.92) (1.45, 9.67) (1.60, 11.7) (2.00, 7.37) (2.15, 13.2)
			};
			
			\addplot[
			only marks, mark=asterisk, mark size=2.9pt,
			teal, line width=0.9pt
			] coordinates {
				(1.00, 3.91) (1.15, 5.35) (1.30, 6.95) (1.45, 8.89) (1.60, 11.3) (1.75, 12.9) (2.00, 6.04) (2.15, 9.38) (2.30, 13.2)
			};
		\end{axis}
\end{tikzpicture}
    
\end{center}
\caption{Maximal SDP block sizes, number of SDP constraints and total computation time versus $(p,\,q,\,r,\,s)$ for the $2$-D torus grid quartic with product cyclic symmetry.}
\label{fig_2D_torus}

\textcolor{black}{$\times$ 
dense \cite{lasserre01}}\\
\textcolor{violet}{\scalebox{1}{$+$} 
sign symmetry}\\
\textcolor{teal}{\scalebox{1.5}{$\ast$} 
$\mathfrak{C}_{p}\times\mathfrak{C}_{q}$-symmetry \cite{riener2013exploiting}}\\
\textcolor{black}{$\blacksquare$ 
term sparsity (maximal chordal extension, diagonal squares) \cite{magron21}}\\ 
\textcolor{red}{\scalebox{1.5}{$\bullet$}
$\mathfrak{C}_{p}\times\mathfrak{C}_{q}$-symmetry + term sparsity (maximal chordal extension, diagonal squares)}\\
\textcolor{blue}{\scalebox{1.1}{$\blacklozenge$}
$\mathfrak{C}_{p}\times\mathfrak{C}_{q}$-symmetry + term sparsity (maximal chordal extension, no diagonal squares)}

\end{figure}
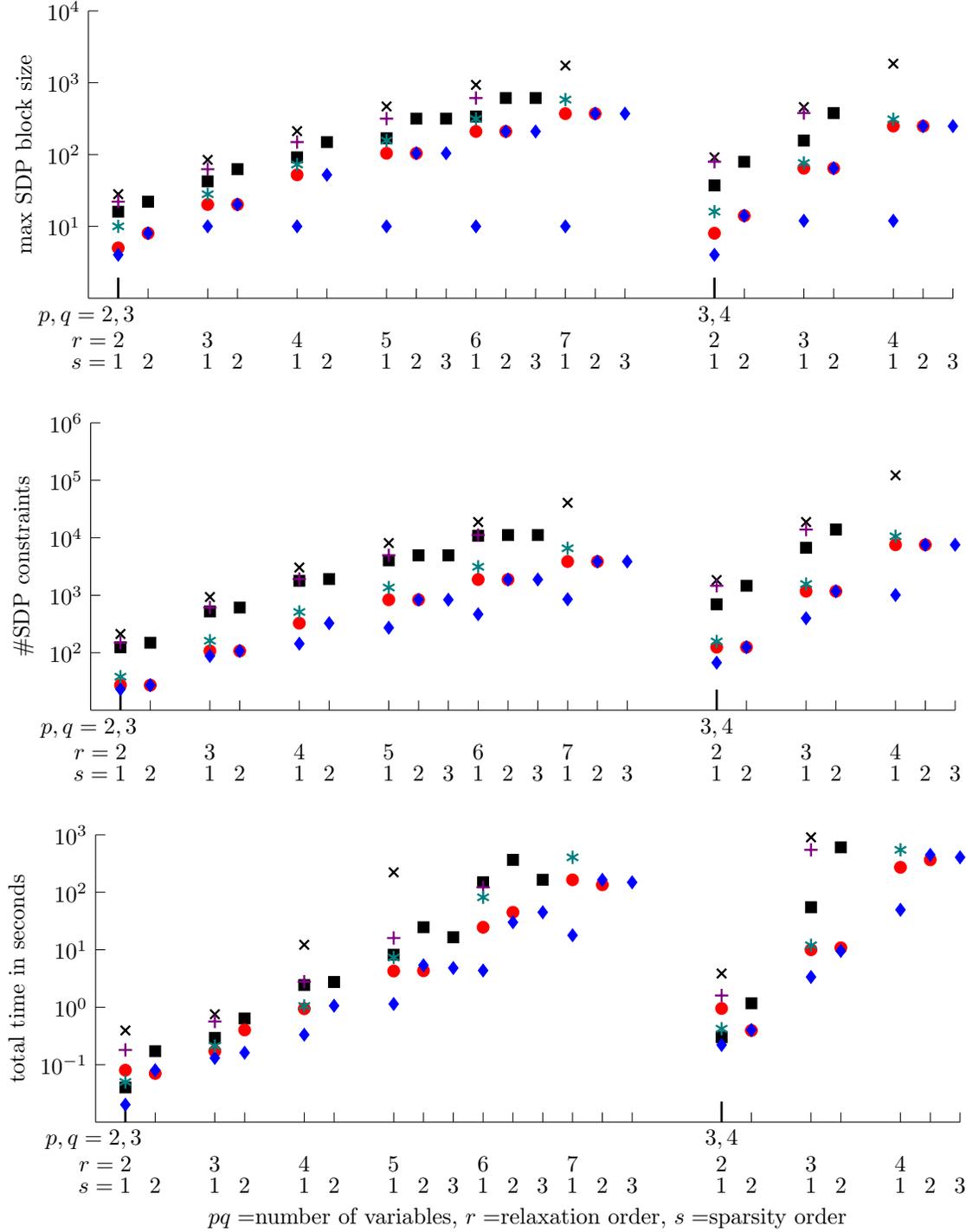

\begin{figure}
\begin{center}
    
\begin{tikzpicture}
\begin{axis}[
			width=0.9\textwidth, height=6cm,
			xmin=0.95, xmax=3.11,
            xlabel={$n=$number of variables, $r=$relaxation order, $s=$sparsity order},
			xtick={
            1.00,1.11,
            1.33,1.44,
            1.66,1.77,
            2.00,2.11,2.22,
            2.33,2.44,
            2.66,2.77,
            3.00,3.11
            },
            xticklabels={
            {\shortstack{$n=           5\phantom{=n}$\\$r= 2\phantom{=r}$\\$s=1\phantom{=s}$}},
                {\shortstack{$\phantom{5}$\\$\phantom{     2}$\\$             2$}},
            {\shortstack{$    \phantom{5}           $\\$   3            $\\$  1$}},
                {\shortstack{$\phantom{5}$\\$\phantom{     3}$\\$             2$}},
            {\shortstack{$    \phantom{5}           $\\$   4            $\\$  1$}},
                {\shortstack{$\phantom{5}$\\$\phantom{     4}$\\$             2$}},
            {\shortstack{$             6            $\\$   2            $\\$  1$}},
                {\shortstack{$\phantom{6}$\\$\phantom{     2}$\\$             2$}},
                {\shortstack{$\phantom{6}$\\$\phantom{     2}$\\$             3$}},
            {\shortstack{$    \phantom{6}\phantom{=n}$\\$  3            $\\$  1$}},
                {\shortstack{$\phantom{6}$\\$\phantom{     3}$\\$             2$}},
            {\shortstack{$    \phantom{6}            $\\$  4            $\\$  1$}},
                {\shortstack{$\phantom{6}$\\$\phantom{     4}$\\$             2$}},
            {\shortstack{$             7            $\\$   2            $\\$  1$}},
                {\shortstack{$\phantom{7}$\\$\phantom{     2}$\\$             2$}},
            },
            extra x ticks={1.00,2.00,3.00},
            extra x tick labels={$\phantom{2}$,$\phantom{2}$,$\phantom{2}$},
            extra x tick style={
                major tick length=9pt,
                tick style={line width=1pt},
                xticklabel style={yshift=-2ex, font=\small, color=black}
            },
            ylabel={max SDP block size},
            ymin=0.0, ymax=6.9, 
			ytick={2.30,4.60,6.90,9.20,11.5,13.8,16.1},
            yticklabels={$10^1$, $10^2$, $10^3$, $10^4$, $10^5$, $10^6$, $10^7$},
			axis x line*=bottom,
			axis y line*=left,
			enlargelimits=false,
			legend cell align=left,
			legend pos=north west,
			tick style={black},
			]
			
			\addplot[
			only marks, mark=square*, mark size=2.4pt,
			draw=black, fill=black
			] coordinates {
				(1.00, 3.04) (1.33, 4.03) (1.66, 4.84) (2.00, 3.33) (2.33, 4.43) (2.66, 5.35) (3.00, 3.58) 
			};
			
			\addplot[
			only marks, mark=square*, mark size=2.4pt,
			draw=black, fill=white
			] coordinates {
				(1.00, 2.40) (1.11, 2.56) (1.33, 2.40) (1.44, 2.56) (1.66, 3.04) (1.77, 3.04) (2.00, 2.77) (2.11, 2.94) (2.33, 2.77) (2.44, 2.94) (2.66, 3.33) (2.77, 3.33) (3.00, 3.09) (3.11, 3.26) 
			};
			
			\addplot[
			only marks, mark=*, mark size=2.6pt,
			draw=red, fill=red
			] coordinates {
				(1.00, 1.39) (1.33, 1.95) (1.66, 2.64) (2.00, 1.39) (2.33, 1.95) (2.66, 2.64) (3.00, 1.39) 
			};
			
			\addplot[
			only marks, mark=o, mark size=2.4pt,
			draw=red, fill=white
			] coordinates {
				(1.00, 1.10) (1.11, 1.39) (1.33, 1.61) (1.44, 1.61) (1.66, 2.20) (1.77, 2.20) (2.00, 1.10) (2.11, 1.39) (2.33, 1.61) (2.44, 1.61) (2.66, 2.20) (2.77, 2.20) (3.00, 1.10) (3.11, 1.39) 
			};
			
			\addplot[
			only marks, mark=diamond*, mark size=2.8pt,
			draw=blue, fill=blue
			] coordinates {
				(1.00, 1.10) (1.11, 1.39) (1.33, 1.61) (1.44, 1.95) (1.66, 1.95) (1.77, 2.64) (2.00, 1.10) (2.11, 1.39) (2.33, 1.61) (2.44, 1.95) (2.66, 1.95) (2.77, 2.64) (3.00, 1.10) (3.11, 1.39) 
			};
			
			\addplot[only marks, mark=diamond*, mark size=2.8pt,
			draw=blue, fill=white
			] coordinates {
				(1.00, 0.693) (1.11, 1.39) (1.33, 1.10) (1.44, 1.61) (1.66, 1.10) (1.77, 2.20) (2.00, 1.10) (2.11, 1.39) (2.33, 1.10) (2.44, 1.61) (2.66, 1.10) (2.77, 2.20) (3.00, 0.693) (3.11, 1.39) 
			};
			
			\addplot[
			only marks, mark=asterisk, mark size=2.8pt,
			coral, line width=0.6pt
			] coordinates {
				(2.00, 1.10) (2.11, 1.10) (2.22, 1.39) 
			};
		\end{axis}
\end{tikzpicture}

\end{center}
\caption{Maximal SDP block sizes versus $(n,\,r,\,s)$ for the symmetric quartic.}
\label{fig_daniel_vs_pop}

\end{figure}
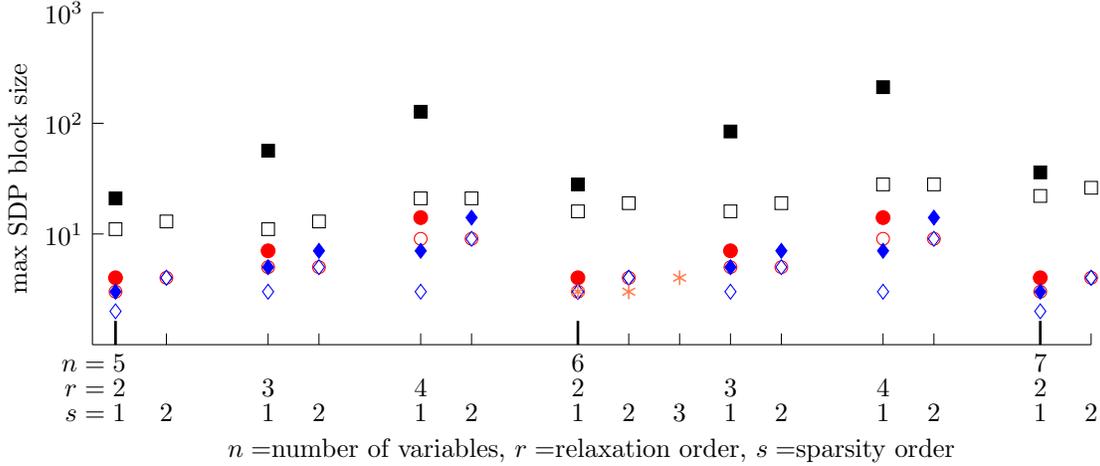

\begin{table}[H]
\begin{center}
\begin{tabular}{cccc}
method 
& sparsity order $s$ &   SDP block sizes                   &      lower bound  \\
\hline
\hline
\textcolor{black}{$\blacksquare$} 
& $1$ &   $28_1$                          &       $-0.5680$   \\
\hline
\textcolor{black}{$\square$} 
& $1$ &   $7_2, 11_5, 16_1$                 &       $-0.7269$   \\
& $2$ &   $7_1, 12_5, 16_1, 17_1, 19_2$     &       $-0.5680$   \\
\hline
\textcolor{red}{\scalebox{1.5}{$\bullet$}} 
& $1$ &   $1_1, 3_1, 4_1$           &       $-0.5680$   \\
\hline
{\greencell}\textcolor{red}{\scalebox{1.5}{$\circ$}} 
{\greencell}& {\greencell}$1$ &   {\greencell}$1_1, 2_2, 3_2$                   &       {\greencell}$-0.5680$   \\
& $2$ &   $1_1, 3_1, 4_1$           &       $-0.5680$   \\
\hline
\textcolor{blue}{\scalebox{1.1}{$\blacklozenge$}} 
& $1$ &   $1_3, 2_1, 3_1$                   &       infeasible  \\
& $2$ &   $1_1, 3_1, 4_1$           &       $-0.5680$   \\
\hline
\textcolor{blue}{\scalebox{1.1}{$\lozenge$}} 
& $1$ &   $1_3, 2_3$                        &       slow progress   \\
& $2$ &   $1_1, 3_1, 4_1$           &       $-0.5680$   \\
\hline
\textcolor{coral}{\scalebox{1.5}{$\ast$}} 
& $1$ &   $1_5, 2_2, 3_1$                   &       $-0.7581$   \\
{\greencell} & {\greencell}$2$ &   {\greencell}$1_2, 2_2, 3_2$                   &       {\greencell}$-0.5680$   \\
& $3$ &   $1_1, 3_1, 4_1$           &       $-0.5680$   \\
\end{tabular}
\caption{
Lower bounds and SDP block sizes for the symmetric quartic, $n=6,\,r=2$. 
}
\label{table_daniel_vs_pop_n6_r2}
\end{center}

\textcolor{black}{$\blacksquare$ 
term sparsity (maximal chordal extension, diagonal squares) \cite{magron21}}
\\
\textcolor{black}{$\square$ 
term sparsity (approximately smallest chordal extension, diagonal squares)}
\\
\textcolor{red}{\scalebox{1.5}{$\bullet$}
$\mathfrak{S}_{n}$-symmetry + term sparsity (maximal chordal extension, diagonal squares)}
\\
\textcolor{red}{\scalebox{1.5}{$\circ$}
$\mathfrak{S}_{n}$-symmetry + term sparsity (approximately smallest chordal extension, diagonal squares)}
\\
\textcolor{blue}{\scalebox{1.1}{$\blacklozenge$}
$\mathfrak{S}_{n}$-symmetry + term sparsity (maximal chordal extension, no diagonal squares)}
\\
\textcolor{blue}{\scalebox{1.1}{$\lozenge$}
$\mathfrak{S}_{n}$-symmetry + term sparsity (approximately smallest chordal extension, no diagonal squares)}
\\
\textcolor{coral}{\scalebox{1.5}{$\ast$} 
term sparsity with $\mathfrak{S}_{n}$-symmetric chordal extension \cite{JohannesThesis}}

\end{table}


\section{Concluding Remarks and Further Investigation Tracks}
\label{section_correlative_symmetry}
We have designed a Lasserre-type moment-sums of squares hierarchy \Cref{eq_symmetric_tssos_hiersarchy} of lower bounds for the polynomial optimization problem \ref{eq_pop} with invariant objective function and constraints. 
The hierarchy is a generalization of the initial TSSOS hierarchy \cite{magron21} in the sense that we construct a tsp graph; see \Cref{defi_tsp_matrix}, but with the novelty that the nodes are elements of a symmetry-adapted basis instead of the standard monomial one. 
In \Cref{thm_sparsity_convergence_1,thm_sparsity_convergence_2}, 
we prove convergence under the maximal chordal extension assumption and Archimedeanity of the quadratic module. 

Our benchmarks in \Cref{section_4_benchmarks} indicate that this new method is superior to both pure term sparsity and pure symmetry exploitation; see \Cref{fig_1D_ring,fig_2D_torus,fig_daniel_vs_pop}. 
The compared parameters are SDP block sizes, numbers of constraints, and computational time. 
The latter strongly depends on the implementation. 
The computation of a symmetry-adapted basis is independent of the objective function and constraints. 
For large groups (such as the symmetric group $\mathfrak{S}_n$ with $n\geq 10$), the package \href{https://github.com/kalmarek/SymbolicWedderburn.jl?tab=readme-ov-file}{\tt SymbolicWedderburn.jl} used to compute symmetry-adapted bases can cause wrong rank estimates in the projection matrices. 

Computation of the block structure of and assembling the SDP also depend on the implementation and can improve the computation time further. 
The key step here is to compute the products of symmetry-adapted basis elements under the Reynolds operator. 

Finally, we conclude with some remarks on variations of the problem and future research directions. 

\subsection{Constraints forming an Orbit}
Throughout 
\Cref{section_2_symmetry_reduction,section_3_symmetry_and_sparsity,section_4_benchmarks}, 
we have assumed $g_k\in\RX^\lineargroup$. 
When we assume instead that the constraints $g_k$ form an orbit, 
then the feasible region $K$ is still $\lineargroup$-stable. 
Indeed, \ref{eq_pop} can be rewritten as 
\begin{equation}\label{eq_pop2}\tag{POP2} 
    f^* 
\coloneqq  \min\{ f(X)\,\vert\,X\in K \} 
    \quad \mbox{with} \quad 
    K
\coloneqq  \{X\in\R^n\,\vert\, 
    \forall \sigma\in \lineargroup: 
    g^\sigma(X)\geq 0\}, 
\end{equation}
where $g\in\{g_1,\ldots,g_\ell\}$ 
of degree $d\coloneqq d_1=\cdots=d_\ell$ 
can be chosen freely among the constraints. 
The assumption $f\in\RX^\lineargroup$ remains unchanged. 


For a $\lineargroup$-invariant linear form $L$, 
the truncated moment matrix $\gmomentmatrix{L}{r}$ 
is still block diagonal with respect to a 
$\lineargroup$-symmetry-adapted basis, 
but the same does not necessarily hold for the localizing matrices
$\gmomentmatrix{g^\sigma*L}{r-d}$. 
Instead, we consider the stabilizer subgroup 
$\lineargroup_k 
\coloneqq \mathrm{Stab}_\lineargroup(g_k) 
\coloneqq \{ \sigma\in\lineargroup\,\vert\,g_k^\sigma=g_k \}$. 
Since the orbits $g_k^\lineargroup$ are equal, 
the $\lineargroup_k$ are conjugate subgroups of $\lineargroup$. 
In particular, the respective isotypic decompositions 
feature the same number of irreducible representations, 
dimensions, and multiplicities. 
Fix a $1\leq k\leq \ell$ and let 
\begin{equation}\label{eq_isotypic_decomposition_stab}
    \RX \hookrightarrow \CX
=   \bigoplus\limits_{i=1}^{\tilde{\irreps}}
    \bigoplus\limits_{j\in \tilde{J}^{(i)}}
    \tilde{W}^{(i)}_j
\end{equation}
be an isotypic decomposition with respect to $\lineargroup_k$
and denote by $\mathcal{R}^{\lineargroup_k}$ the Reynolds operator. 
Analogously to \Cref{prop_symmetry_adapted_basis}, 
we can compute a $\lineargroup_k$-symmetry-adapted 
basis for $\RX_{2(r-d)}$ 
with elements denoted by $\tilde{w}^{(i)}_{1,j}$. 
Note that any $\lineargroup$-invariant linear form $L$ 
is in particular $\tilde{\lineargroup}$-invariant and thus
the localized $\tilde{\lineargroup}$-symmetry-adapted 
moment matrix is now block diagonal 
as in \Cref{eq_moment_matrix_blocks}. 
The entries, indexed by 
$1\leq j,j'\leq \tilde{m}^{(i)}_{r-d}$, are 
\begin{equation}\label{eq_stabilizer_moment_matrix_entries}
    (\tilde{\mathbf{M}}^{(i)}_{r-d}(g_k*L))_{j,j'}
=   L\left(
    \frac{1}{\vert\lineargroup_k\vert}
    \sum\limits_{\sigma\in\lineargroup_k}
    (g_k\,\tilde{w}^{(i)}_{1,j}\,\tilde{w}^{(i)}_{1,j'})^\sigma
    \right)  
\end{equation}
and can be expressed in terms of the $\lineargroup$-moments $\gmomentsequence$ of $L$. 

\subsection{Sparsity after Orbit Space Reduction}

To compute the global minimum of an invariant polynomial $f\in\RX^\lineargroup$, our symmetry reduction method relies on the fact that the polynomial ring has an isotypic decomposition and the moment matrix is block diagonal in a symmetry-adapted basis. 

Alternatively, Hilbert's Finiteness Theorem states that the invariant ring 
$\RX^\lineargroup$ is finitely generated as an $\R$-algebra \cite{Hilbert90}, 
that is, there is a \textbf{set of fundamental invariants} 
$\{p_1,\ldots,p_m\}\subseteq \RX^\lineargroup$ with the following property: 
For any $f\in\RX^\lineargroup$, there exists some $m$-variate polynomial 
$g\in\R[{z}_1,\ldots,{z}_m]\eqqcolon\RZ$ such that $f = g(p_1,\ldots,p_m)$. 
Denote by $p: \R^n \to \R^m,\,X\mapsto (p_1(X),\ldots,p_m(X))$ the \textbf{Hilbert map}. 
The image $p(\R^n)$ is called the \textbf{orbit space of $\lineargroup$}. 

There are various methods to compute fundamental invariants, for example, 
Molien's Formula (\cite[Sec.~2.1]{gatermann2000}, \cite[Alg.~2.2.5]{Sturmfels08}), 
primary and secondary invariants (\cite[Sec.~2.5]{Sturmfels08}, \cite[Sec.~3.5--7]{DerksenKemper15}), or 
Derksen's \cite{Derksen99} or King's \cite{King13} algorithm.

\begin{theorem}[Procesi \& Schwarz \cite{procesischwarz85}]\label{thm_procesi_schwarz}
There exists a matrix polynomial ${\posmat}\in\R[{z}]^{m\times m}$ with the following property: 
For $Z\in\R^m$, we have $Z\in p(\R^n)$ if and only if $\posmat(Z)\succeq 0$. 
\end{theorem}

In particular, the unconstrained polynomial optimization problem 
\ref{eq_pop} with $K=\R^n$ can be rewritten as 
\begin{equation}\label{eq_pop3}\tag{POP3} 
    f^*
=   \min\{f(X)\,\vert\,X\in\R^n\}
=   \min\{g(Z)\,\vert\,Z\in\R^m,\, \mathbf{P}(Z)\succeq 0\} . 
\end{equation}
A promising investigation track would be to leverage the recent work \cite{WangMillerGuo25} that addresses sparsity exploitation for polynomial optimization problems with polynomial matrix constraints such as \ref{eq_pop3}. 

\subsection{Equivariant Dynamical Systems}

Finally, we explore the idea of exploiting both symmetry and term sparsity for computing a maximal positive invariant (MPI) set of a dynamical system. 
Let $\lineargroup\subseteq\mathrm{GL}_n(\R)$ be a finite group 
and $f=\transpose{({f}_1,\ldots,{f}_n)}\in\RX^{n}$ 
be a \textbf{$\lineargroup$-equivariant} polynomial vector field, that is, 
\begin{equation}
    \sigma^{-1} \cdot f 
=   f \circ \sigma^{-1} 
=   ({f}_1^\sigma,\ldots,{f}_n^\sigma)^T 
\eqqcolon  f^\sigma 
    \mbox{ whenever } \sigma \in \lineargroup . 
\end{equation}
Here, the left-hand side is given by matrix-vector multiplication and 
the right-hand side by the induced action of $\lineargroup$ on the polynomial entries. 
For $X : \R\to\R^n,\,t\mapsto \transpose{(X_{1}(t),\ldots,X_{n}(t))}$ differentiable 
and $g_1,\ldots,g_\ell\in\RX^\lineargroup$, 
consider the \textbf{equivariant dynamical system}
\begin{equation}
    \dot{X}(t)
=   f(X(t)),
    \quad \mbox{s.t.} \quad
    \forall\,1\leq k\leq \ell:\,
    g_k(X(t)) \geq 0 . 
\end{equation}
The symmetry-adapted moment-SOS approximation for the 
\textbf{maximum positively invariant} (MPI) set 
of order $r$ with discount factor $\beta > 0$ is 
\begin{equation}
    \begin{array}[t]{rl} 
    \sup        &   \gmoment{1} \\ 
    \mbox{s.t.} &   \gmomentsequence, \tilde{\gmomentsequence}, \hat{\gmomentsequence} 
                    \in \R^{\gmultiplicity{1}{2r}} ,\, \forall \, j : \\ 
                &   \gmoment{j} = 
                    z_j - \tilde{y}_j \\ 
                &   \gmoment{j} = 
                    \beta \, \hat{y}_j - \sum_{j'} f_{j,j'} \, \hat{y}_{j'} , \\
                &   \momentblock{g_k*       \gmomentsequence }{i}{r-d_k} \succeq 0 , \\ 
                &   \momentblock{g_k*\tilde{\gmomentsequence}}{i}{r-d_k} \succeq 0 , \\ 
                &   \momentblock{g_k*  \hat{\gmomentsequence}}{i}{r-d_k} \succeq 0 
    \end{array}
    \, \mbox{and} \,\, 
    \begin{array}[t]{rl}
    \inf  \,\,      &   \sum_j h_j\,z_j \\
    \mbox{s.t.} &   h\in\RX^\lineargroup_{2r} ,\, 
                    v\in\RX^\lineargroup_{2r+1-d_f} , \\
                &   q^{(i)}_k , \tilde{q}^{(i)}_k , \hat{q}^{(i)}_k 
                    \in \soscone{i}{r-d_k}, \\ 
                &   \begin{array}[t]{r}
                        h
                    =   \sum_j h_j\,w^{(1)}_j 
                    =   \sum_{k,i} \reynolds(q^{(i)}_k\,g_k), \\ 
                        h - v - 1 
                    =   \sum_{k,i} \reynolds(\tilde{q}^{(i)}_k\,g_k), \\ 
                        \beta v - \transpose{\nabla v} \cdot \mathbf{S} \cdot f 
                    =   \sum_{k,i} \reynolds(\hat{q}^{(i)}_k\,g_k), 
                    \end{array}
    \end{array}
\end{equation}
where $f_{j,j'} \in \R$ are certain coefficients of a term defined in \Cref{lemma_equivariance} and $z_j$ are the $\lineargroup$-moments of the Lebesgue measure; 
see also \cite{WangSchlosserKordaMagron23} for the non-symmetric version. 
Analogously to \Cref{defi_tsp_matrix}, 
one can then construct a tsp graph for the MPI approximation and 
exploit term sparsity in the symmetry-adapted basis. 
Future research efforts include a practical implementation of these latter hierarchies on concrete dynamical system applications.  

\section*{Acknowledgments}

This work was supported by 
the European Union's HORIZON–MSCA-2023-DN-JD programme under the Horizon Europe (HORIZON) Marie Sk{\l}odowska-Curie Actions under the grant agreement 101120296 (TENORS), 
the AI Interdisciplinary Institute ANITI funding through the French ``France 2030'' program under the Grant agreement n°ANR-23-IACL-0002, 
the project COMPUTE funded within the QuantERA II programme that has received funding from the EU's H2020 research and innovation programme under the GA No.~101017733 {\normalsize\euflag}, 
the National Key R\&D Program of China under grant No.~2023YFA1009401, 
the Natural Science Foundation of China under grant No.~12571333, 
and the International Partnership Program of Chinese Academy of Sciences under grant No.~167GJHZ2023001FN. 
IK was also supported by the Slovenian Research Agency program P1-0222 and grants 
J1-50002, N1-0217, J1-60011, J1-50001, J1-3004 and J1-60025, 
and partially supported by the Fondation de l'Ecole polytechnique as part of the Gaspard Monge Visiting Professor Program. 
IK thanks Ecole Polytechnique and Inria for hospitality during the preparation of this manuscript.

\bibliographystyle{alpha}

\bibliography{refer}

\appendix

\section{Linear Representation Theory}
\label{appendix_rep_theo}
Let $\lineargroup$ be a finite group. 
All $\lineargroup$-modules and representations are defined over the complex numbers $\C$. 

\begin{theorem}[Maschke's Theorem]\cite[Ch.1,~Thm.2]{serre77}
Every $\lineargroup$-module can be decomposed into 
a direct sum of irreducible $\lineargroup$-submodules. 
\end{theorem}

\begin{proposition}[Schur's Lemma]\cite[Ch.2,~Prop.4]{serre77}
Let $W,\tilde{W}$ be two irreducible $\lineargroup$-modules. 
Then any $\phi\in\mathrm{Hom}_\lineargroup(W,\tilde{W})$ 
is either the zero map or an isomorphism. 
If $W,\tilde{W}$ are finite dimensional, 
then $\mathrm{Hom}_\lineargroup(W,\tilde{W})$ 
has vector space dimension $1$. 
\end{proposition}





\begin{proof}[Proof of \Cref{thm_symm_moment_sos_hierarchy_SDP}]
For $f^r_{\mathrm{mom}}$, this is clear. 
If $\sosmatrix{i}{k}$ is feasible for the SDP, 
it is in particular orthogonally diagonalizable with nonnegative eigenvalues. 
Denote by $\mathbf{w}^{(i)}_{1}$ the vector of basis elements $w^{(i)}_{1,j}$ with 
$1\leq j\leq \gmultiplicity{i}{r-d_k}$. 
Then 
$q^{(i)}_k \coloneqq \transpose{(\mathbf{w}^{(i)}_1)} \cdot \sosmatrix{i}{k} \cdot \mathbf{w}^{(i)}_1$ 
is a sum of squares of elements in $\rspan{\mathcal{S}^{(i)}_{r-d_k}}$ with 
\begin{align*}
    \sum_{k,i} \reynolds (q^{(i)}_k\,g_k)
&=  \sum_{k,i} \reynolds (\transpose{(\mathbf{w}^{(i)}_1)} \cdot 
    \sosmatrix{i}{k} \cdot \mathbf{w}^{(i)}_1\,g_k) \\
&=  \sum_{k,i} \trace(\underbrace{\reynolds
    (\mathbf{w}^{(i)}_1 \cdot \transpose{(\mathbf{w}^{(i)}_1)}\,g_k)
    }_{=\sum_j \sdpmatrix{i}{r,k,j}\,w^{(1)}_{1,j}} \cdot \sosmatrix{i}{k})\\
&=  \sum_j \sum_{k,i} \trace(\sdpmatrix{i}{r,k,j} \cdot \sosmatrix{i}{k})\,w^{(1)}_{1,j}\\
&=  f_1-(\underbrace{f_1-\sum_{k,i} \trace(\sdpmatrix{i}{r,k,1} \cdot 
    \sosmatrix{i}{k})}_{\eqqcolon\lambda}) + 
    \sum_{j\geq 2} \underbrace{\sum_{k,i} \trace(\sdpmatrix{i}{r,k,j} \cdot 
    \sosmatrix{i}{k})}_{=f_j}\,w^{(1)}_{1,j}\\
&=  f-\lambda.
\end{align*}
Hence, $\lambda$ and the $q^{(i)}_k$ are feasible for $f^r_{\mathrm{sos}}$. 
Conversely, let $\lambda$ be feasible for $f^r_{\mathrm{sos}}$ with sums of squares 
\[
    q^{(i)}_k
=   \sum_{t} (q^{(i)}_{k,t})^2
    \quad \mbox{such that}  \quad
    f-\lambda 
=   \sum_{k,i}
    \reynolds (q^{(i)}_k\,g_k) .
\]
For $q\in\rspan{\mathcal{S}^{(i)}_{r-d_k}}$, 
denote by $\mathbf{vec}(q)$ the coefficient vector. 
Then
\begin{align*}
    \sum_j f_j\,w^{(1)}_{1,j} - \lambda
=   f-\lambda
&=  \sum_{k,i}
    \reynolds
    \left(\sum_t \transpose{(\mathbf{vec}(q^{(i)}_{k,t})} \cdot 
    \mathbf{w}^{(i)}_1 )^2\,g_k\right) \\
&=  \sum_{k,i} \sum_t
    \transpose{\mathbf{vec}(q^{(i)}_{k,t})}\cdot
    \underbrace{
    \reynolds (
    \mathbf{w}^{(i)}_{1}\cdot\transpose{(
    \mathbf{w}^{(i)}_{1}
    )}\,g_k
    )}_{=\sum_j A^{(i)}_{r,k,j}\,w^{(1)}_{1,j}}
    \cdot
    \mathbf{vec}(q^{(i)}_{k,t})\\
&=  \sum_{k,i} \trace 
    \left( \resizebox{0pt}{11pt}{$\sum\limits_{k,i}$} \right.
    \sum_j (\sdpmatrix{i}{r,k,j} \, w^{(1)}_{1,j}) \cdot
    \underbrace{\sum_t
    \mathbf{vec}(q^{(i)}_{k,t})\cdot\transpose{(
    \mathbf{vec}(q^{(i)}_{k,t})
    )}}_{\eqqcolon\sosmatrix{i}{k}}
    \left. \resizebox{0pt}{11pt}{$\sum\limits_{k,i}$} \right) .
\end{align*}
Comparing coefficients yields 
$f_1-\lambda=\sum_{k,i} \trace(\sdpmatrix{i}{r,k,1}\cdot \sosmatrix{i}{k})$ and 
$f_j=\sum_{k,i}\trace(\sdpmatrix{i}{r,k,j}\cdot \sosmatrix{i}{k})$ for 
$2\leq j \leq \gmultiplicity{1}{2r}$. 
This completes the proof. 
\end{proof}

\section{Equivariants}
\label{appendix_equivariants}


\begin{lemma}\label{lemma_equivariance}
If $p\in\RX$ is $\lineargroup$-invariant, 
then the formal gradient 
$\transpose{\nabla p=(\frac{\partial p}{\partial X_1},\ldots,\frac{\partial p}{\partial X_n})}$ 
is $\lineargroup$-equivariant. 
Furthermore, if $\mathbf{f},\mathbf{g}\in\RX^n$ are 
$\lineargroup$-equivariant and $\mathbf{S}\in\R^{n\times n}$ 
satisfies $\mathbf{S}=\transpose{\sigma}\cdot \mathbf{S} \cdot \sigma$ 
for all $\sigma\in\lineargroup$, then 
$\transpose{\mathbf{f}} \cdot \mathbf{S} \cdot \mathbf{g} \in \RX$ 
is $\lineargroup$-invariant. 
\end{lemma}
\begin{proof}
Let $\sigma\in\lineargroup$. 
By $\lineargroup$-invariance and the chain rule, we have
\[
    (\nabla p)(\var{})
=   (\nabla(p\circ\sigma))(\var{})
=   \sigma\cdot(\nabla p)(\sigma(\var{}))
\]
and thus $\sigma^{-1}\cdot\nabla p=(\nabla p)^\sigma$. 
Furthermore, by linearity and $\lineargroup$-equivariance, we have
\[
    (\transpose{\mathbf{f}} \cdot \mathbf{S} \cdot \mathbf{g})^\sigma
=   \transpose{\mathbf{f}} \cdot \transpose{(\sigma^{-1})} 
    \cdot \mathbf{S} \cdot \sigma^{-1} \cdot \mathbf{g}
=   \transpose{\mathbf{f}} \cdot \mathbf{S} \cdot \mathbf{g}. 
\]
\end{proof}

\begin{remark}
One way to construct a matrix $\mathbf{S}$ that satisfies 
the hypothesis of \Cref{lemma_equivariance} is to set
$\mathbf{S}\coloneqq\frac{1}{\vert\lineargroup\vert}
\sum_{\sigma\in\lineargroup}\transpose{\sigma}\cdot\sigma$. 
If $\lineargroup$ is an orthogonal group, 
then one may choose $\mathbf{S}$ as the identity matrix. 
\end{remark}

\begin{remark}
The matrix polynomial $\mathbf{P}\in\RZ^{m\times m}$ from 
\Cref{thm_procesi_schwarz} can be constructed as follows: 
Let $\lineargroup$ be a finite group with fundamental invariants $\{p_1,\ldots,p_m\}$, 
$\mathbf{S}\in\R^{n\times n}$ be as in \Cref{lemma_equivariance}, 
and $\tilde{\posmat}\in\RX^{m\times m}$ 
be the matrix with entries 
$\tilde{\posmat}_{ij} = \transpose{(\nabla p_i)}\cdot \mathbf{S} \cdot \nabla p_j$. 
Since $\tilde{\posmat}_{ij}\in\RX^\lineargroup$ by \Cref{lemma_equivariance}, 
there exists a matrix
${\posmat}\in\R[{z}]^{m\times m}$ such that
$\posmat(p_1(\var{}),\ldots,p_m(\var{})) = \tilde{\posmat}(\var{})$. 
\end{remark}

\Addresses

\end{document}